\documentclass{amsart}

\usepackage{amsmath,amsfonts,amssymb,amsthm,pb-diagram,picinpar,graphicx,color}
\usepackage{colortbl}
\usepackage{enumerate}

\newcommand{\bC}{\mathbb{C}}
\newcommand{\bR}{\mathbb{R}}
\newcommand{\bZ}{\mathbb{Z}}
\newcommand{\bP}{\mathbb{P}}
\newcommand{\calL}{\mathcal{L}}
\newcommand{\calE}{\mathcal{E}}
\newcommand{\calO}{\mathcal{O}}

\newcommand{\fd}{\mathfrak{d}}
\newcommand{\fm}{\mathfrak{m}}
\newcommand{\I}{\mathrm{I}}
\newcommand{\pb}{\mathcal{P}}
\newcommand{\pbp}{\varphi}

\newtheorem{Th}{Theorem}[section]
\newtheorem{Prop}[Th]{Proposition}
\newtheorem{Lem}[Th]{Lemma}
\newtheorem{Cor}[Th]{Corollary}
\newtheorem{Rem}[Th]{Remark}
\newtheorem{Ex}[Th]{Example}
\newtheorem{Def}[Th]{Definition}

\begin{document}
\title{
Nodal curves with a contact-conic and Zariski pairs
}
\date{}
\author{Shinzo Bannai}
\address{Department of Natural Sciences, National Institute of Technology, Ibaraki College, 866 Nakane, Hitachinaka, Ibaraki 312-8508, Japan}
\email{sbannai@ge.ibaraki-ct.ac.jp}

\author{Taketo Shirane}
\address{National Institute of Technology, Ube College, 2-14-1 Tokiwadai, Ube, Yamaguchi 755-8555, Japan}
\email{tshirane@ube-k.ac.jp}



\maketitle

\begin{abstract}
In this present paper, we study the splitting of nodal plane curves with respect to contact conics. We define the notion of {\it splitting type} of such curves and show that it can be used as an invariant to distinguish the embedded topology of plane curves. We also give a criterion to determine the splitting type in terms of the configuration of the nodes and tangent points. As an application, we construct sextics and contact conics with prescribed splitting types, which give rise to new Zariski-triples. 
 
\end{abstract}

\section{Introduction}

Let $\Gamma$ be a irreducible curve on the complex projective plane $\bP^2$, 
and let $\pi:X\to\bP^2$ be a Galois cover branched along $\Delta\subset\bP^2$. 
We call $\Gamma$ a \textit{splitting curve} with respect to $\pi$ if the pull back $\pi^{\ast}\Gamma$ of $\Gamma$ is reducible. 
Splitting curves have been studied by various mathematicians. 
For example, E.~Artal-Bartolo and Tokunaga studied the splitting of nodal rational curves with respect to a double cover branched along a smooth conic in \cite{k-plet}; 
I.~Shimada studied the splitting of smooth conics with respect to a double cover branched along a sextic curve in \cite{shimada}; 
the first author studied the splitting type of pairs of smooth curves of degree up to three with respect to a double cover branched along nodal quartics in \cite{bannai}; and 
the second author studied the splitting number of a smooth curve for a simple cyclic cover branched along a smooth curve in \cite{shirane}. 
In the papers listed above, the splitting of rational curves and smooth curves is intensively studied. 
However, it seems that there were few studies of the splitting of non-rational singular curve. 
For example, Tokunaga studied the splitting of nodal quartic curves in \cite{tokunaga1, tokunaga2}. 
In this paper, we investigate the splitting type of non-rational nodal curves with respect to a double cover branched along a smooth conic. 
More precisely, following the study of \cite{k-plet}, we define the splitting type of a plane curve with respect to  a smooth conic.

One of the motivation of study of the splitting curves is its application to Zariski $k$-plets, 
where 
a \textit{Zariski $k$-plet} is a $k$-plet $(C_1,\dots, C_k)$ of plane curves $C_1,\dots, C_k\subset\bP^2$ satisfying the following two conditions; 
\begin{enumerate}
\item
there exist tubular neighborhoods $\mathcal{T}_i$ of $C_i$ for $i=1,\dots,k$ such that the $k$ topological pairs $(\mathcal{T}_i,C_i)$ are homeomorphic; 
\item 
the topological pairs $(\bP^2,C_i)$ and $(\bP^2,C_j)$ are not homeomorphic if $i\ne j$. 
\end{enumerate}
We call a Zariski $2$-plet a \textit{Zariski pair}. 
It is known that, for a plane curve $C$ and its tubular neighborhood $\mathcal{T}$,  the homeomorphism class of $(\mathcal{T},C)$ is determined by the combinatorial type of $C$, 
where the combinatorial type of $\Gamma$ is data of number of its irreducible components, degree and singularities of each component, configuration of the components (see \cite{survey} for detail). 
There are may results about Zariski pairs by using some topological invariants of plane curves (cf. \cite{artal, acc, survey, k-plet, bannai, degtyarev, degtyarev2, benoit-meilhan, oka, shimada, shirane, zariski}). 
A basic invariant is the fundamental group of the complement of a plane curve (cf. \cite{artal, survey, k-plet, oka, zariski}). 
Recently, Zariski pairs which can not be distinguished by the fundamental group were found (cf. \cite{survey, degtyarev2, shirane}). 
We call such a Zariski pair a \textit{$\pi_1$-equivalent Zariski pair}. 
In particular, the result of \cite{shirane} shows that study of the splitting curves allow us to construct $\pi_1$-equivalent Zariski pairs. 
In this paper, we 
prove that our splitting type is a topological invariant, 
and can be used to distinguish topology of plane curves. 

The main theorem of this paper is a criterion for determining  the splitting type of a nodal curve with respect to a smooth conic in terms of the  configuration of the nodes and tangent points of the nodal curve and the smooth conic (Theorem~\ref{thm. split}). 
However, for a given splitting type, it is difficult to explicitly construct nodal curves with a contact conic which have the splitting type, especially when the nodes and tangent points need to be in general position. 
In order to construct various examples of nodal sextic curves with a contact conic, 
we prove that there exists a correspondence between nodal sextic curves on $\bP^2$ with a contact conic and quartic surfaces in $\bP^3$ with a node, where a node of a surface means an $A_1$-singularity. 

Nodal quartic surfaces have been studied classically. 
In \cite{jessop}, a classification of nodal quartic surfaces was given. Also in \cite{coble}, quartic surfaces with $8$-nodes were studied from a different perspective.
An $8$-nodal quartic surface is called \textit{syzygetic} if the quartic surface is defined by a quadratic form of quadratic forms, 
otherwise it is call \textit{asyzygetic} (see Definition~\ref{def. conical} for details). 
It is known that the geometric difference of syzygetic and asyzygetic quartic surface lie in the configuration of their nodes. 
In this paper, we prove that a syzygetic quartic surface and an asyzygetic quartic surface give two sextic curves with contact conics in the correspondence mentioned above, which have different splitting types. 
Moreover, we improve our criterion for splitting types of $7$-nodal sextic curves by difference of syzygetic and asyzygetic quartic surfaces (Theorem~\ref{thm. conical}). 

This paper is organized as follows: 

In Section~\ref{sec. splitting}, we define the splitting type of plane curves with respect to a smooth conic, and give a criterion for determining the splitting type in terms of the position of singularities. 
In Section~\ref{sec. correspondence}, we give a correspondence between quartic surfaces with a node and sextic curves with a contact conic. 
In Section~\ref{sec. quartic surfaces}, we recall the geometric difference of syzygetic and asyzygetic quartic surfaces in detail in order to improve the criterion  in Section~\ref{sec. splitting} for $7$-nodal sextic curves with a contact conic. 
In Section~\ref{sec. improve}, we improve the criterion in Section~\ref{sec. splitting} for $7$-nodal sextic curves. 
In the final section, we give new examples of Zariski pairs consisting of $6$ and $7$-nodal sextic curves with a contact conic. 
\medskip

\section{Splitting curves with respect to contact conics}\label{sec. splitting}

The aim of this section is to prove Theorem~\ref{thm. split} below. 
To state Theorem~\ref{thm. split}, we define some notion. 

\begin{Def}{\rm
Let $\Gamma$ be a curve. 
We call $\Gamma$ a \textit{nodal curve} if all of the singular points of $\Gamma$ are nodes. 
Moreover, for a positive integer $r$, we call a nodal curve $\Gamma$ an \textit{$r$-nodal curve} if the number of nodes of $\Gamma$ is $r$. Note that we allow $\Gamma$ to be reducible unless otherwise stated.
}\end{Def}

For two curves $\Gamma_1, \Gamma_2\subset\bP^2$ and a point $P\in\Gamma_1\cap\Gamma_2$, 
the local intersection number of $\Gamma_1$ and $\Gamma_2$ at $P$ is denoted by $\I_P(\Gamma_1,\Gamma_2)$. 
For a curve $\Delta\subset\bP^2$ of degree $2 d$, 
let $\pi_{\Delta}:X_{\Delta}\to\bP^2$ denote the double cover branched along $\Delta$, 
and let $\iota_{\Delta}:X_{\Delta}\to X_{\Delta}$ be the covering transformation of $\pi_{\Delta}$. 
Note that, if $\Delta$ is a smooth conic, then $X_{\Delta}\cong\bP^1\times\bP^1$ and $\iota_{\Delta}$ interchanges the two rulings. 

\begin{Def}{\rm
Let $\Gamma$ be a curve on $\bP^2$. 
\begin{enumerate}
\item
A smooth conic $\Delta\subset\bP^2$ is called a \textit{contact conic of $\Gamma$} 
if all of the intersection points of $\Gamma$ and $\Delta$ are smooth points of $\Gamma$, and $I_P(\Gamma,\Delta)\geq 2$ for all $P\in \Gamma\cap\Delta$. 
Furthermore, a contact conic $\Delta\subset\bP^2$ of $\Gamma$ is called an \textit{even contact conic} (resp. a \textit{simple contact conic}) of $\Gamma$ 
if  $I_P(\Gamma,\Delta)$ is even (resp. equal to two) for all $P\in\Gamma\cap\Delta$. 
\item 
Let $\Gamma\subset\bP^2$ be a curve with an even contact conic $\Delta$. 
For two positive integers $m\leq n$, 
we call $\Gamma$ a \textit{splitting curve of type $(m,n)$ with respect to $\Delta$} 
if $\pi_{\Delta}^{\ast}\Gamma=D^++D^-$ for an $(m,n)$-divisor $D^+$ on $X_{\Delta}\cong\bP^1\times\bP^1$, where $D^-=\iota_{\Delta}^{\ast}D^+$. 
We call $\Gamma$ a \textit{non-splitting curve with respect to $\Delta$} 
if $\Gamma$ is not a splitting curve of type $(m,n)$ with respect to $\Delta$ for any $m\leq n$, or equivalently, 
$\pi_{\Delta}^{\ast}\Gamma_0$ is irreducible for some irreducible component $\Gamma_0$ of $\Gamma$. 
\end{enumerate}
}\end{Def}

\begin{Th}
For $i=1,2$, let $\Gamma_i\subset\bP^2$ be an irreducible curve of degree $d\geq 3$ with a contact conic $\Delta_i$. 
Assume that there exists a homeomorphism $h:\bP^2\to\bP^2$ such that $h(\Gamma_1+\Delta_1)=\Gamma_2+\Delta_2$. 
Then $\Gamma_1$ is a splitting curve of type $(m,n)$ with respect to $\Delta_1$ if and only if  $\Gamma_2$ is also a splitting curve of type $(m,n)$ with respect to $\Delta_2$. 
\end{Th}
\begin{proof}
Since $\deg\Gamma_i=d\geq 3$, we have $h(\Gamma_1)=\Gamma_2$ and $h(\Delta_1)=\Delta_2$. 
Let $\bar{\Delta}_i\subset X_{\Delta_i}$ be the preimage of $\Delta_i$ for $\pi_{\Delta_i}$. 
As in the proof of \cite[Proposition~2.5]{bannai}, there exists a homeomorphism $\bar{h}:X_{\Delta_1}\setminus\bar{\Delta}_1\to X_{\Delta_2}\setminus\bar{\Delta}_2$ such that $h\circ\pi_{\Delta_1}=\pi_{\Delta_2}\circ \bar{h}$ over $\bP^2\setminus\Delta_2$: 
\[
\begin{diagram}
\node{X_{\Delta_1}\setminus\bar{\Delta}_1} \arrow{s,l}{\pi_{\Delta_1}} \arrow{e,t}{\bar{h}} \node{X_{\Delta_2}\setminus\bar{\Delta}_2} \arrow{s,r}{\pi_{\Delta_2}} \\
\node{\bP^2\setminus\Delta_1} \arrow{e,b}{h} \node{\bP^2\setminus\Delta_2}
\end{diagram}
\]
Suppose that $\Gamma_1$ is a splitting curve of type $(m,n)$ with respect to $\Delta_1$, say $\pi_{\Delta_1}^{\ast}\Gamma_1=D_1^++D_1^-$ for a $(m,n)$-divisor $D^+$. 
Let $D_2^+$ and $D_2^-$ be the closures of $\bar{h}(D_1^+\setminus\bar{\Delta}_1)$ and $\bar{h}(D_1^-\setminus\bar{\Delta}_1)$, respectively. 
Then $D_2^+$ and $D_2^-$ are irreducible components of $\pi_{\Delta_2}^{\ast}\Gamma_2$. 
Hence $\Gamma_2$ is a splitting curve with respect to $\Delta_2$, say of type $(m',n')$. 
Since the local intersection number of $D_i^+$ and $D_i^-$ at a point $\bar{P}$ of $\pi_{\Delta_i}^{-1}(\Delta_i\cap\Gamma_i)$ determines the intersection of $\Gamma_i$ and $\Delta_i$ at $\pi(\bar{P})$, we obtain 
\[ m^2+n^2=D_1^+. D_1^-=D_2^+.D_2^-=m'^2+n'^2. \]
Since $m+n=m'+n'=d$, $m\leq n$ and $m'\leq n'$, we obtain $m=m'$ and $n=n'$. 
\end{proof}

For  curves $\Gamma$, $\Delta$ and a line $L$ on $\bP^2$, we let $\gamma$, $\delta$ and $l$ be homogeneous polynomials which define $\Gamma$, $\Delta$ and $L$, respectively. For a real number $r\in\bR$, the maximal integer not beyond $r$ is denoted by $[r]$. 
In the following, we say that  $L$ is  {\it a general line} (with respect to $\Gamma$) if  $L$ intersects $\Gamma$ transversally.

\begin{Th}\label{thm. split}
Let $m,n$ be two integers with $0<m\leq n$, $d:=m+n$, $k:=n-m$, 
and $\alpha:=(m^2+n^2-m-n)/2$. 
Let $\Gamma\subset\bP^2$ be a nodal curve of degree $d$ with a simple contact conic 
$\Delta$. 
Let $T_1,\dots,T_{d}$ be the tangent points of $\Gamma$ and $\Delta$. 

Then the following three conditions are equivalent: 
\begin{enumerate}[\rm (i)]
\item\label{thm. split 1} 
$\Gamma$ is a splitting curve of type $(m,n)$ with respect to $\Delta$; 
\item\label{thm. split 2} 
for any general line $L$ tangent to $\Delta$, there exist $c_i\in H^0(\bP^2,\calO_{\bP^2}(i))$ for $i=n-1, n$ satisfying the following conditions:  
\begin{enumerate}[{\rm ({ii}-a)}]
\item 
$c_n$ and $c_{n-1}$ have no common factor with $l$, and 
\item 
$\gamma\, l^k=c_n^2-\delta\,c_{n-1}^2$; 
\end{enumerate}
\item\label{thm. split 3}
for a general line $L$ tangent to  $\Delta$, there exist $c_i\in H^0(\bP^2,\calO_{\bP^2}(i))$ for $i=n-1, n$ satisfying the conditions (ii-a) and (ii-b) above; 
\item\label{thm. split 4} 
for any general  line $L$ tangent to  $\Delta$ at $T_0\in\Delta$, 
there are $\alpha$ nodes  $P_1,\dots, P_{\alpha}\in\Gamma$ of $\Gamma$, $m$ intersection points $Q_1,\dots, Q_m\in\Gamma\cap L$ of $\Gamma$ and $L$, and two divisors $C_n$ and $C_{n-1}$ on $\bP^2$ of degree $n$ and $n-1$ respectively
satisfying the following conditions; 
\begin{enumerate}[{\rm ({iv}-a)}]
\item both  $C_n$ and $C_{n-1}$ pass through the nodes $P_1,\dots,P_{\alpha}$, and 
$C_n$ passes through the tangent points $T_1,\dots,T_{d}\in\Gamma\cap \Delta$; 
\item $C_{j}$ is smooth at $Q_i$, 
and $\I_{Q_i}(C_{j},\Gamma)=k$ for each $i=1,\dots,m$ and $j=n-1,n$; 
\item $C_n$ has $[k/2]+(1-(-1)^k)/2$ local branches at $T_0$, and $[k/2]$ of the local branches are tangent to $\Delta$; 
\item $C_{n-1}$ has $[(k-1)/2]+(1-(-1)^{k-1})/2$ local branches at $T_0$, 
and $[(k-1)/2]$ of the local branches are them tangent to $\Delta$; 
\item any two of the five divisors $\Gamma$, $\Delta$, $L$, $C_n$ and $C_{n-1}$ have no common components; 
\end{enumerate}
\item\label{thm. split 5} 
for a general  line $L$ tangent to $\Delta$ at $T_0\in\Delta$, 
there are $\alpha$ nodes $P_1,\dots, P_{\alpha}\in\Gamma$ of $\Gamma$,  $m$ intersection points $Q_1,\dots, Q_m\in\Gamma\cap L$ of $\Gamma$ and $L$, and two divisors $C_n$ and $C_{n-1}$ on $\bP^2$ of degree $n$ and $n-1$ respectively
satisfying the conditions in (iv-a), \dots, (iv-e) above; 
\end{enumerate}
\end{Th}

%

\begin{proof}
We first prove the equivalence of the conditions (\ref{thm. split 1}), (\ref{thm. split 2}) and (\ref{thm. split 3}). 
Assume that $\pi_{\Delta}^{\ast}\Gamma=D^++D^-$ for some $(m,n)$-divisor $D^+$ on $\bP^1\times\bP^1$ and $D^-:=\iota_{\Delta}^{\ast}D^+$. 
Let $L$ be a general tangent line of $\Delta$. 
Since $L$ is a rational curve and tangent to $\Delta$, 
$\pi_{\Delta}^{\ast}L=L^++L^-$ for some $(1,0)$-divisor $L^+$ and $L^-:=\iota_{\Delta}^{\ast}L^+$. 
Hence $\Gamma+k\, L$ is split into two $(n,n)$-divisors by $\pi_{\Delta}$, i.e. 
\[ \pi_{\Delta}^{\ast}(\Gamma+k\, L)=(D^++k\, L^+)+(D^-+k\, L^-). \]
Let $\tilde{d}^+=0$ be a defining equation of the $(n,n)$-divisor $D^++k\, L^+$, 
and $\tilde{d}^-=\iota_{\Delta}^{\ast}\tilde{d}^+$. 
Since $\tilde{d}^++\tilde{d}^-$ is invariant under $\iota_{\Delta}^{\ast}$, 
there exists $c_n\in H^0(\bP^2,\calO_{\bP^2}(n))$ such that 
$\pi_{\Delta}^{\ast}c_{n}=\tilde{d}^++\tilde{d}^-$. 
Let $R\subset\bP^1\times\bP^1$ be the ramification locus of $\pi_{\Delta}$, 
and let $r=0$ be the defining equation of $R$. 
Since $\tilde{d}^+-\tilde{d}^-=0$ along $R$, $\tilde{d}^+-\tilde{d}^-$ is divided by $r$. 
Moreover, since $\iota_{\Delta}^{\ast}r=-r$, $(\tilde{d}^+-\tilde{d}^-)/r$ is invariant under $\iota_{\Delta}^{\ast}$. 
Thus there is $c_{n-1}\in H^0(\bP^2,\calO_{\bP^2}(n-1))$ 
such that $\pi_{\Delta}^{\ast}c_{n-1}=(\tilde{d}^+-\tilde{d}^-)/r$. 
Since $\pi_{\Delta}^{\ast}(\Gamma+k\,L)$ is defined by $\pi_{\Delta}^{\ast}(c_n^2-\delta\,c_{n-1}^2)=0$, 
 condition (\ref{thm. split 2}) holds. 

It is trivial that  condition (\ref{thm. split 2}) implies condition (\ref{thm. split 3}). 
Now, assume that condition (\ref{thm. split 3}) holds. 
Then we have $\pi_{\Delta}^{\ast}(\Gamma+k\, L)=\widetilde{D}^++\widetilde{D}^-$, 
where $\widetilde{D}^{\pm}$ are the $(n,n)$-divisors on $\bP^1\times\bP^1$ defined by 
$\pi_{\Delta}^{\ast}c_n\pm r\,\pi_{\Delta}^{\ast}c_{n-1}=0$. 
By  condition (\ref{thm. split 2}-a), $\pi_{\Delta}^{\ast}c_n\pm r\,\pi_{\Delta}^{\ast}c_{n-1}$ are not divided by $\pi_{\Delta}^{\ast}l$ 
since $\pi_{\Delta}^{\ast}c_n- r\,\pi_{\Delta}^{\ast}c_{n-1}=\iota_{\Delta}^{\ast}(\pi_{\Delta}^{\ast}c_n+ r\,\pi_{\Delta}^{\ast}c_{n-1})$ and 
$\iota_{\Delta}^{\ast}\pi_{\Delta}^{\ast}l=\pi_{\Delta}^{\ast}l$. 
Since $\Gamma$ and $\pi_{\Delta}^{\ast}\Gamma$ are reduced, 
$\widetilde{D}^{\pm}=D^{\pm}+k\, L^{\pm}$ for some $(m,n)$-divisor $D^+$ and $D^-=\iota_{\Delta}^{\ast}D^+$. 
Therefore we obtain $\pi_{\Delta}^{\ast}\Gamma=D^++D^-$, 
and  conditions (\ref{thm. split 1}), (\ref{thm. split 2}) and (\ref{thm. split 3}) are equivalent. 

Next we prove that  condition (\ref{thm. split 1}) implies  condition (\ref{thm. split 4}). 
Suppose that $\pi_{\Delta}^{\ast}\Gamma=D^++D^-$ for an $(m,n)$-divisor $D^+$ and $D^-:=\iota_{\Delta}^{\ast}D^+$, and 
let $L$ be a general tangent line of $\Delta$. 
Since the intersection number $D^+.D^-$ is equal to $m^2+n^2$, and 
$D^+\cap D^-\cap R=\pi_{\Delta}^{-1}(\Gamma\cap\Delta)$ consists of $d=m+n$ points, 
there are $\alpha$ nodes of $\Gamma$, $P_1,\dots,P_{\alpha}$, such that 
\[ \{P_1,\dots,P_{\alpha},T_1,\dots,T_d\}=\pi_{\Delta}(D^+\cap D^-). \]
Since $L^+.D^-=m$, there are $Q_1,\dots,Q_m\in \Gamma$ such that 
\[ \{Q_1,\dots,Q_m\}=\pi_{\Delta}(L^+\cap D^-). \]
By (\ref{thm. split 2}), we have $\gamma\, l^k=c_n^2-\delta\, c_{n-1}^2$. 
Let $C_n$ and $C_{n-1}$ be the divisors on $\bP^2$ given by $c_n=0$ and $c_{n-1}=0$, respectively. 
We prove that $P_i$, $Q_j$, $C_n$ and $C_{n-1}$ satisfy (\ref{thm. split 4}-a), \dots, (\ref{thm. split 4}-e).

As in the above argument, $\tilde{d}^{\pm}=\pi_{\Delta}^{\ast}c_n\pm r\,\pi_{\Delta}^{\ast}c_{n-1}=0$ define 
$D^{\pm}+k\, L^{\pm}$, respectively. 
This implies that $\pi_{\Delta}^{\ast}c_n$ and $r\,\pi_{\Delta}^{\ast}c_{n-1}$ vanish at $(D^++L^+)\cap (D^-+L^-)$. 
Hence $C_n$ (resp. $C_{n-1}$) passes through $\pi_{\Delta}(D^+\cap D^-)$ (resp. $P_1,\dots,P_{\alpha}$). 
Thus $C_n$ and $C_{n-1}$ satisfy the condition (\ref{thm. split 4}-a). 

At an intersection $P\in D^+\cap L^-$,  we can write locally $\tilde{d}^+=s$ and $\tilde{d}^-=t^k$ since $L$ intersects transversally with $\Gamma$, 
where $s$ and $t$ are generators of the maximal ideal $\fm_P$ at $P$. 
Hence $\pi_{\Delta}^{\ast}C_{n}$ and $\pi_{\Delta}^{\ast}C_{n-1}$ are defined by $s+t^k=0$ and $s-t^k=0$ respectively. 
Since $P\in\bP^1\times\bP^1$ is an unramified point of $\pi_{\Delta}$, 
$C_n$ and $C_{n-1}$ satisfies condition (\ref{thm. split 4}-b).  

At the intersection $P$ of $L^+$ and $L^-$, 
we can write locally $\tilde{d}^+=s^k$ and $\tilde{d}^-=t^k$. 
Therefore, we have 
{\allowdisplaybreaks
\begin{eqnarray*}
\tilde{d}^++\tilde{d}^-&=&\left\{
\begin{array}{ll} 
(s+t){\displaystyle \prod_{i=1}^{(k-1)/2}}(s^2-(\zeta_{2k}^{2i-1}+\zeta_{2k}^{2k-2i+1})\,s\, t+t^2) & \mbox{if $k$ is odd} \\
{\displaystyle \prod_{i=1}^{k/2}}(s^2-(\zeta_{2k}^{2i-1}+\zeta_{2k}^{2k-2i+1})\,s\, t+t^2) & \mbox{if $k$ is even,}
 \end{array}
\right. \\
\tilde{d}^+-\tilde{d}^- &=&
\left\{\begin{array}{ll} 
(s-t){\displaystyle \prod_{i=1}^{(k-1)/2}}(s^2-(\zeta_{k}^{i}+\zeta_{k}^{k-i})\,s\, t+t^2) & \mbox{if $k$ is odd} \\
(s-t)(s+t){\displaystyle \prod_{i=1}^{(k-2)/2}}(s^2-(\zeta_{k}^{i}+\zeta_{k}^{k-i})\,s\, t+t^2) & \mbox{if $k$ is even,}
 \end{array}
\right. 
\end{eqnarray*}
}
where $\zeta_k$ is a primitive $k$th root of unity. 
Since $s-t=0$ defines $R$ at $P$,  conditions (\ref{thm. split 4}-c) and (\ref{thm. split 4}-d) in the assertion hold. 
It is clear that any two of $\tilde{d}^+$, $r$, $l^-$, $\tilde{d}^++\tilde{d}^-$ and $(\tilde{d}^+-\tilde{d}^-)/r$ have no common factors, 
where $l^-=0$ is a defining equation of $L^-$. 
Therefore condition (\ref{thm. split 4}-e) in the assertion holds. 
Hence  condition (\ref{thm. split 1}) implies condition (\ref{thm. split 4}). 

It is trivial that condition (\ref{thm. split 4}) implies condition (\ref{thm. split 5}). 
Finally, we prove that  condition (\ref{thm. split 5}) implies  condition (\ref{thm. split 3}). 
We assume that there exist $\alpha$ nodes $P_1,\dots,P_{\alpha}\in\Gamma$, 
$m$ intersection points $Q_1,\dots,Q_m\in\Gamma\cap L$ and two divisors $C_n, C_{n-1}\subset\bP^2$ satisfying  conditions (\ref{thm. split 4}-a), \dots, (\ref{thm. split 4}-e). 
Let $\gamma, \delta, l, c_n$ and $c_{n-1}$ be homogeneous polynomials defining $\Gamma, \Delta, L, C_n$ and $C_{n-1}$, respectively. 
Let $\Lambda$ be the linear system generated by three divisors $D_1:=\Gamma+k\, L, D_2:=2\,C_n$ and $D_3:=2\,C_{n-1}+\Delta$. 
It is sufficient to prove $\dim\Lambda=1$. 
We seek the base points of $\Lambda$ 
in order to resolve the indeterminacy locus of the rational map $\Phi_{\Lambda}:\bP^2\dashrightarrow\bP^N$ ($N=\dim\Lambda$) given by $\Lambda$. 
Note that, by condition (\ref{thm. split 4}-e), the number of base points of $\Lambda$ is finite. 
From  conditions (\ref{thm. split 4}-a) and (\ref{thm. split 4}-b), 
\begin{eqnarray*}
\Gamma.C_n &\geq& \sum_{i=1}^{\alpha}\I_{P_i}(\Gamma,C_n)+\sum_{i=1}^{m}\I_{Q_i}(\Gamma,C_n)+\sum_{i=1}^{d}\I_{T_i}(\Gamma,C_n) \\
&\geq& 2\,\alpha+m\,k+d \\
&=&n(m+n). 
\end{eqnarray*}
Hence we have $\Gamma\cap C_n=\{P_1,\dots,P_{\alpha}, Q_1,\dots,Q_m,T_1,\dots,T_d\}$, 
$\I_{P_i}(\Gamma,C_n)=2$ ($1\leq i\leq\alpha$) and $\I_{T_i}(\Gamma,C_n)=1$ ($1\leq i \leq d$). 
Moreover, $C_n$ is smooth at $P_i$ ($1\leq i\leq \alpha$) and $T_i$ ($1\leq i\leq d$). 
By the same argument, we have 
\begin{eqnarray*}
\Gamma\cap C_{n-1} &=& \{P_1,\dots,P_{\alpha},Q_1,\dots,Q_m\}, \\
L\cap C_n&=& L\cap C_{n-1} \ \  = \ \  \{Q_1,\dots,Q_m,T_0\}, \\
\Delta\cap C_n &=& \{T_0,T_1,\dots,T_d\}, \\
C_n\cap C_{n-1} &=& \{ P_1,\dots,P_{\alpha},Q_1,\dots,Q_m, T_0 \}. 
\end{eqnarray*}
Thus we obtain $D_i\cap D_{j}=\{P_1,\dots,P_{\alpha},Q_1,\dots,Q_m,T_0,\dots,T_d\}$ for $1\leq i<j\leq 3$. 
Moreover, the following conditions hold; 
{\allowdisplaybreaks
\begin{itemize}
\item[($\mathrm{P}_i$)] local branches of $\Gamma$, $C_n$ and $C_{n-1}$ at $P_i$ are smooth, and intersect transversally with each other; 
\item[($\mathrm{Q}_i$)] $\Gamma$, $C_n$ and $C_{n-1}$ intersect each other at $Q_i$ with multiplicity $k$; 
\item[($\mathrm{T}_0$)] local branches of $C_n$ and $C_{n-1}$ at $T_0$ are smooth, and 
$L$, $\Delta$ and local branches of $C_j$ ($j=n,n-1$) at $T_0$ intersect each other with multiplicity at most $2$; 
\item[($\mathrm{T}_i$)] $\Gamma$ and $C_n$ intersect transversally at $T_i$ ($1\leq i\leq d$). 
\end{itemize}
}
We resolve the  indeterminacy locus of $\Phi_{\Lambda}$ through blowing-ups. 
For a blowing-up $\sigma:\widetilde{X}\to X$, by abuse of notation, we use the same symbol $D$ for a divisor $D$ on $X$ to describe the pull back $\sigma^{\ast}D$. 
Let $\sigma_{P_i}:X_{P_i}\to\bP^2$ be the blowing-up at $P_i\in\bP^2$, and $E_{P_i}$ the exceptional divisor of $\sigma_{P_i}$. 
Since $m_{P_i}(D_j)=2$ for $i=1,2,3$, the divisors $
D_j-2\,E_{P_i}$ on $X_{P_i}$ have no intersection with each other over $P_i$. 

Let $\sigma_{T_0,1}:X_{T_0,1}\to\bP^2$ be the blowing-up at $T_{0}\in\bP^2$, and 
$E_{T_0,1}$ the exceptional divisor of $\sigma_{T_0,1}$. 
From conditions (\ref{thm. split 4}-c), (\ref{thm. split 4}-d) and ($\mathrm{T}_0$), 
$
D_j-k\, E_{T_0,1}\subset X_{T_0,1}$ ($j=1,2,3$)  intersect at one point, say $T_{0,1}$, 
and the local branches of $
D_j-k\, E_{T_0,1}$ at $T_{0,1}$ intersect transversally with each other. 
Let $\sigma_{T_0,2}:X_{T_0,2}\to X_{T_0,1}$ be the blowing-up at $T_{0,1}\in X_{T_0,1}$, and 
$E_{T_0,2}$ the exceptional divisor of $\sigma_{T_0,2}$. 
Then the divisors $D_j-k\, E_{T_0,1}-k\, E_{T_0,2}\subset X_{T_0,2}$ ($j=1,2,3$) do not intersect each other over $T_0$. 

Let $\sigma_{Q_i,1}:X_{Q_i,1}\to \bP^2$ be the blowing-up at $Q_i\in\bP^2$, and 
let $E_{Q_i,1}$ be the exceptional divisor of $\sigma_{Q_i,1}$. 
By condition (\ref{thm. split 4}-b), we have 
\[ \I_{Q_i}(D_j,D_{j'})=k+1 \ (j\ne j'), \ \ m_{Q_i}(D_1)=k, \ \  m_{Q_i}(D_j)=2 \  (j=2,3). \]
From conditions (\ref{thm. split 4}-b) and ($\mathrm{Q}_i$), the divisors $D_j-2\,E_{Q_i,1}$ on $X_{Q_i,1}$ intersect at one point, say $Q_{i,1}$. 
Moreover, we have  
\begin{eqnarray*}
 \I_{Q_{i,1}}(D_j-2\,E_{Q_i,1}, D_{j'}-2\,E_{Q_i,1})&=&k-1, \\
m_{Q_{i,1}}(D_1-2\,E_{Q_i,1})&=&k,  \\
m_{Q_{i,1}}(D_j-2\,E_{Q_i,1})&=&2 \ \ \ \ \  (j=2,3). 
\end{eqnarray*}
Similarly, let $\sigma_{Q_i,j}:X_{Q_i,j}\to X_{Q_i,j-1}$ ($j=2,\dots,k$) be such blowing-ups over $Q_i$, 
and let $E_{Q_i,j}$ be the exceptional divisors of $\sigma_{Q_i,j}$. 
Then the three divisors $D_j-\sum_{j=1}^{k}2\, E_{Q_i,j}$ do not intersect each other over $Q_i$. 

Let $\sigma_{T_i,1}:X_{T_i,1}\to \bP^2$ be the blowing-up at $T_i\in\bP^2$ for $1\leq i\leq d$, and 
let $E_{T_i,1}$ be  the exceptional divisor of $\sigma_{T_i,1}$. 
By condition ($T_i$), $D_j-E_{T_i,1}$ ($j=1,2,3$) intersect at one points, say $T_{i,1}$, 
and intersect transversally with each other. 
Let $\sigma_{T_i,2}:X_{T_i,2}\to X_{T_i,1}$ be the blowing-up at $T_{i,1}$, and 
let $E_{T_i,2}$ be the exceptional divisor of $\sigma_{T_i,2}$. 
Then $D_j-E_{T_i,1}-E_{T_i}$ ($j=1,2,3$) do not intersect each other over $T_i$. 

Let $\sigma:\widetilde{X}\to\bP^2$ be the composition of blowing-ups $\sigma_{P_i}$, 
$\sigma_{Q_i,j}$ ($j=1,\dots,k$), $\sigma_{T_0,j}$ ($j=1,2$), $\sigma_{T_i,j}$ ($j=1,2$), 
and 
\[ \widetilde{D}_j:=D_j-\sum_{i=1}^{\alpha}2\,E_{P_i}-\sum_{i=1}^{m}\sum_{j'=1}^{k}2\,E_{Q_i,j'}-k\,(E_{T_0,1}+E_{T_0,2})-\sum_{i=1}^d(E_{T_i,1}+E_{T_i,2}). \]
Put $\widetilde{\Lambda}$ as the linear system on $\widetilde{X}$ generated by $\widetilde{D}_j$ ($j=1,2,3$). 
Then the self-intersection number of $\widetilde{D}_j$ is equal to $0$, $\widetilde{D}_j^2=0$. 
Hence we have $\dim\mathrm{Im}(\Phi_{\widetilde{\Lambda}})=1$. 
Moreover, since $\widetilde{D}_j.E_{T_i,2}=1$ for $j=1,\dots,d$, 
we have $\deg\Phi_{\widetilde{\Lambda}}(\widetilde{D}_j)=1$. 
Therefore $\dim\Lambda=\dim\widetilde{\Lambda}=1$. 
\end{proof}

From Theorem~\ref{thm. split}, we have the following corollary. 

\begin{Cor}\label{lem. number of nodes}
Let $\Gamma$ be a $r$-nodal curve of degree $d$ with a simple contact conic $\Delta$. 
If $\Gamma$ is a splitting curve of type $(m,n)$ with respect to $\Delta$, 
then $2\,r\geq  m^2+n^2-d$. 
\end{Cor}

Theorem~\ref{thm. split} enables us to determine the splitting type of a given curve $\Gamma$ with respect to a simple contact conic $\Delta$ by geometrical conditions (\ref{thm. split 4}-a), \dots, (\ref{thm. split 4}-e). 
However, the conditions (\ref{thm. split 4}-b), (\ref{thm. split 4}-c), (\ref{thm. split 4}-d) seems complicated.

The corollary below is useful to prove that a curve $\Gamma$ is not a splitting curve of type $(m,n)$ with respect to a simple contact conic $\Delta$. We prepare some terminology used in the statement. Let $X$ be a nonsingular variety, $P_1,\dots,P_r\in X$ be distinct $r$ points of $X$ and  $D\subset X$ be a divisor of $X$. 
We denote the  linear system consisting of effective divisors of $|D|$ through $P_1,\dots,P_r$ by $|D-P_1-\dots-P_r|$.
The points $P_1,\dots,P_r$ are called the \textit{assigned base points} of $\fd:=|D-P_1-\dots-P_r|$. 
Let $\sigma:\widetilde{X}\to X$ be the blowing-up at $P_1,\dots,P_r$, and 
$E_1,\dots,E_r$ the exceptional divisors of $\sigma$. 
We call base points of $|\sigma^{\ast}D-E_1-\dots-E_r|$, considered as infinitely near points of $X$, \textit{unassigned base points} of $\fd$.

\begin{Cor}\label{cor. dimension}
Let $\Gamma$ be a nodal curve of degree $d$ with a simple contact conic $\Delta$, 
and let $T_1,\dots,T_d$ be tangent points of $\Gamma$ and $\Delta$. 
Assume that $\Gamma$ is a splitting curve of type $(m,n)$ with respect to $\Delta$. 
Put $\alpha=(m^2+n^2-m-n)/2$. 
Then there are $\alpha$ nodes of $\Gamma$, $P_1,\dots,P_{\alpha}$, 
such that 
$\fd:=| n\,L-P_1-\dots-P_{\alpha}-T_1-\dots-T_d |$ has no unassigned base points, and 
$\dim\fd\geq n-m$, 
where $L$ is a line on $\bP^2$. 
\end{Cor}
\begin{proof}
By the assumption, $\pi_{\Delta}^{\ast}\Gamma=D^++D^-$ for some $(m,n)$-divisor $D^+$ and $D^-=\iota_{\Delta}^{\ast}D^+$. 
Let $k:=n-m$ and $L_1,\dots,L_k$ be general  lines tangent to $\Delta$. 
Then $\widetilde{\Gamma}:=\Gamma+L_1+\dots+L_k$ is a splitting curve of type $(n,n)$ with respect to $\Delta$. 
By Theorem~\ref{thm. split}, there exist $n^2-n$ nodes of $\widetilde{\Gamma}$, $P_1,\dots,P_{n^2-n}$, and a divisor $C_n$ of degree $n$ which has no common components with $\Delta$ and passes through $P_1,\dots, P_{n^2-n}$ and all tangent points of $\widetilde{\Gamma}$ and $\Delta$. 
As in the proof of Theorem~\ref{thm. split}, we may assume that $\{P_1,\dots,P_{\alpha}\}=\pi_{\Delta}(D^+\cap D^-)\subset \Gamma$. 
Hence $P_1,\dots,P_{\alpha}$ are independent of the  choice of $L_1,\dots,L_k$. 
Thus the generality of $L_1,\dots,L_k$ implies that $\dim| n\,L-P_1-\dots-P_{\alpha}-T_1-\dots-T_d|\geq k=n-m$. 

Let $P\in\bP^2\setminus\{P_1,\dots,P_\alpha\}$ be a point not on $\Delta$. 
Let $L_0$ be a tangent line of $\Delta$ through $P$. 
Put $\pi_{\Delta}^{\ast}L_0=L_0^++L_0^-$, where $L_0^+$ is a $(1,0)$-curve and $L_0^-=\iota_{\Delta}^{\ast}L_0^+$, 
and put $d^{\pm}=0$ and $l_0^{\pm}=0$ as defining equations of $D^{\pm}$ and $L_0^{\pm}$, respectively. 
If $P\in\Gamma$ and $\pi^{-1}(P)=\{P^+,P^-\}$ with $P^+\in D^+$ and $P^-\in D^-$, 
then we can choose $L_0$ such that $P^+\not\in L_0^-$ and $P^-\not\in L_0^+$. 
Then the curve $C_n$ of degree $n$ on $\bP^2$ such that $\pi_{\Delta}^{\ast}C_n$ is defined by $(l_0^+)^k\,d^++(l_0^-)^k\,d^-=0$  passes through  $P_1,\dots,P_{\alpha}, T_1,\dots,T_d$, but not $P$. 
Hence $\fd$ has no unassigned base points. 
\end{proof}

\section{Quartic surfaces and sextic curves with a contact conic}\label{sec. correspondence}

In this section, we prove that there exists a surjection from the set of normal quartic surfaces in $\bP^3$ with a node (i.e. and $A_1$-type singularity) at a point $P\in\bP^3$, which contain no lines through $P$, to the set of sextic curves on $\bP^2$ with an even contact conic (Theorem~\ref{thm. corr}). 
We will apply this result to prove a simple criterion for a nodal sextic curve with a contact conic to be a splitting curve of type $(2,4)$ in Section~\ref{sec. quartic surfaces}. 
In order to prove Theorem~\ref{thm. corr}, we consider a birational map between $\bP^1$-bundles over $\bP^2$ given by Sumihiro \cite{sumihiro}. 

For $i\in\bZ_{\geq 1}$, we put $\calE_i:=\calO_{\bP^2}\oplus\calO_{\bP^2}(-i)$. 
Let $\pbp_i:\pb_i\to\bP^2$ be the $\bP^1$-bundle with $\pb_i:=\bP_{\bP^2}(\calE_i)$, 
and let $\calL_i$ be the tautological bundle on $\pb_i$ with $\pbp_{i \ast}\calL_i\cong\calE_i$. 
Let $(x_0:x_1:x_2)$ be homogeneous coordinates of $\bP^2$, and 
let $U_j$ be the affine open subset of $\bP^2$ given by $x_j\ne 0$ for $j=0,1,2$. 
Let $x_{j,0}$ and $x_{j,i}$ denote the local basis of $\calO_{\bP^2}$ and $\calO_{\bP^2}(-i)$ on $U_j$ for $j=0,1,2$, respectively. 
By abuse of notation, we regard $(x_{j,0}:x_{j,i})$ as homogeneous coordinates of the fibers of $\pbp_i$ over $U_j$. 
Let $A_j$ denote the coordinate ring $H^0(U_j,\calO_{\bP^2}|_{U_j})$ of $U_j$. 
We prove the following lemma by the same argument of \cite[V Proposition~2.6]{hartshorne}. 

\begin{Lem}\label{lem. section}
If $S\subset\pb_i$ is a section of $\pbp_i$, then $\calO_{\pb_i}(S)$ is isomorphic to either $\calL_i$ or $\calL_i\otimes \pbp_i^{\ast}\calO_{\bP^2}(i)$. 
Moreover, 
if $\calO_{\pb_i}(S)\cong\calL_i$, then $S$ is defined by $x_{j,0}=0$ over $U_j$ for $j=0,1,2$; and if $\calO_{\pb_i}(S)\cong\calL_i\otimes \pbp_i^{\ast}\calO_{\bP^2}(i)$, then $S$ satisfies one of the following conditions; 
\begin{enumerate}
\item 
$S$ is defined by $x_{j,i}=0$ over $U_j$; 
\item
there exists an effective divisor $D_S$ of degree $i$ on $\bP^2$ such that $S$ is defined by $g_{j}\, x_{j,0}-x_{j,i}=0$ over $U_j$, 
where $g_j=0$ is a defining equation of $D_S$ over $U_j$. 
\end{enumerate} 
In particular, a section $S_{i,0}$ of $\pbp_i$ satisfying $\calO_{\pb_i}(S_{i,0})\cong\calL_i$ is uniquely determined. 
\end{Lem}
\begin{proof}
By \cite[II Proposition~7.12]{hartshorne}, a section $s:\bP^2\to\pb_i$ of $\pbp_i$ corresponds to a surjection $\eta:\calE_i\to\calO_{\bP^2}(k)$ for some $k\in\bZ$, and 
we have $s^{\ast}\calL_i\cong\calO_{\bP^2}(k)$. 
The kernel of $\eta$ gives the ideal sheaf of the image of $s$ 
since a surjection $\eta:\calE_i\to\calO_{\bP^2}(k)$ induces a surjection $S(\eta):S(\calE_i)\to S(\calO_{\bP^2}(k))$ which defines $s:\bP^2\to\pb_i$, 
where $S(\bullet)$ denotes the symmetric algebra of a vector bundle $\bullet$. 
A homomorphism $\eta:\calE_i\to \calO_{\bP^2}(k)$ corresponds to a piar $(\eta_0,\eta_i)$, 
where $\eta_0$ and $\eta_i$ are global sections of $\calO_{\bP^2}(k)$ and $\calO_{\bP^2}(k+i)$, respectively. 
In particular, $\eta$ is zero at $\eta_0=\eta_i=0$. 
Hence, if $\eta$ is surjective, then $k=-i$ or $k=0$. 

Let $s:\bP^2\to\pb_i$ be a section of $\pbp_i$ given by a surjection $\eta:\calE_i\to\calO_{\bP^2}(k)$, 
and let $S\subset\pb_i$ be the image of $s$. 
We consider the exact sequence 
\[ 0\to\calL_i\otimes\calO_{\pb_i}(-S)\to\calL_i\to\calL_i\otimes\calO_S\to 0. \]
Taking $\pbp_{i \ast}$, we have the following exact sequence 
\begin{equation}\label{eq. exact}
 0\to \pbp_{i \ast}(\calL_i\otimes\calO_{\pb_i}(-S))\to\calE_i\to\calO_{\bP^2}(k)\to 0 
\end{equation}
with $0$ on the right because $R^1\pbp_{i \ast}(\calL_i\otimes\calO_{\pb_i}(-S))=0$ (cf. \cite[V Lemma~2.4]{hartshorne}). 
Since the degree of $\calL_i\otimes\calO_{\pb_i}(-S)$ along  fibers is $0$, 
$\pbp_{i \ast}(\calL_i\otimes\calO_{\pb_i}(-S))$ is invertible by \cite[III Corollary~12.9]{hartshorne}. 
Moreover, we obtain 
\[ \calL_i\otimes\calO_{\pb_i}(-S)\cong \pbp_i^{\ast}\pbp_{i \ast}(\calL_i\otimes\calO_{\pb_i}(-S)). \]

If $k=-i$, then the exact sequence (\ref{eq. exact}) implies $\pbp_{i \ast}(\calL_i\otimes\calO_{\pb_i}(-S))\cong\calO_{\bP^2}$. 
Thus $\calL_i\cong\calO_{\pb_i}(S)$. 
Moreover, since $\eta$ corresponds to $(0,\eta_i)$ for some $\eta_i\in\bC^{\times}$, 
the kernel of $\eta$ is locally generated by $x_{j,0}$ on $U_j$. 
Hence the section $S$ is locally defined by $x_{j,0}=0$ over $U_j$. 

If $k=0$, then the exact sequence (\ref{eq. exact}) implies $\pbp_{i \ast}(\calL_i\otimes\calO_{\pb_i}(-S))\cong\calO_{\bP^2}(-i)$. 
Therefore, $\calO_{\bP_i}(S)\cong\calL_i\otimes \pbp_i^{\ast}\calO_{\bP^2}(i)$. 
Suppose $\eta$ corresponds to $(\eta_0, \eta_i)$ for $\eta_k\in H^0(\bP^2,\calO_{\bP^2}(k))$, 
then $\eta_0\in\bC^{\times}$ since $\eta$ is surjective. 
If $\eta_i=0$, then $\eta|_{U_j}:A_j\,x_{j,0}\oplus A_j\,x_{j,i}\to A_j$ is given by $x_{j,0}\mapsto \eta_0$ and $x_{j,i}\mapsto 0$; 
hence $S$ is locally defined by $x_{j,i}=0$ over $U_j$. 
Suppose that $\eta_i\ne 0$. 
Let $D_S$ be the effective divisor on $\bP^2$ corresponding to $\eta_i\in H^0(\bP^2,\calO_{\bP^2}(i))$. 
In this case, the image of the second factor $\calO_{\bP^2}(-i)$ of $\calE_i$ by $\eta$ is the ideal sheaf of $D_S$. 
Thus $\eta$ is locally given by $x_{j,0}\mapsto \eta_0$ and $x_{j,i}\mapsto \eta_0\,g_j$ over $U_j$, where $g_j=0$ is a defining equation of $D_S$ on $U_j$. 
Therefore $S$ is defined by $g_j\,x_{j,0}-x_{j,i}=0$ over $U_j$. 
\end{proof}

Let $C_d\subset\bP^2$ be a reduced curve of degree $d$ locally defined by $f_j=0$ on $U_j$. 
Let $Y_{i,d}\subset\pb_i$ denote the closed subscheme $S_{i,0}\cap \pbp_i^{\ast}C_d$, 
and let $\sigma_{i,d}:V_{i,d}\to\pb_i$ be the blowing-up along $Y_{i,d}$. 
Put $E_{i,d}$ as the exceptional divisor of $\sigma_{i,d}$. 
By elementary transformations of vector bundles \cite[(1.5) Theorem]{sumihiro}, 
the divisor $(\pbp_i\circ\sigma_{i,d})^{\ast}C_d-E_{i,d}$ is contractible. 
Let $\sigma'_{i,d}:V_{i,d}\to \pb_{i,d}'$ be the contraction of $(\pbp_i\circ\sigma_{i,d})^{\ast}C_d-E_{i,d}$. 
By \cite[(1.5) Theorem]{sumihiro} again, there exists a $\bP^1$-bundle structure $\pbp'_{i,d}:\pb_{i,d}'\to \bP^2$ with $\pbp_i\circ\sigma_{i,d}=\pbp'_{i,d}\circ\sigma'_{i,d}$. 
Hence we obtain the following commutative diagram: 
\begin{equation}\label{diagram} 
\begin{diagram}
\node[2]{V_{i,d}} \arrow{sw,t}{\sigma_{i,d}} \arrow{se,t}{\sigma'_{i,d}} \\
\node{\pb_i} \arrow{se,b}{\pbp_i} \arrow[2]{e,t,..}{\theta_{i,d}} \node[2]{\pb'_{i,d}} \arrow{sw,b}{\pbp'_{i,d}} \\
\node[2]{\bP^2}
\end{diagram} 
\end{equation}

\begin{Prop}\label{prop. birational}
The $\bP^1$-bundle $\pb'_{i,d}$ is isomorphic to $\pb_{i+d}$, and 
$\pbp'_{i,d}:\pb_{i,d}\to\bP^2$ coincides with $\pbp_{i+d}:\pb_{i+d}\to\bP^2$.  
Moreover, the birational map $\theta_{i,d}=\sigma'_{i,d}\circ(\sigma_{i,d})^{-1}:\pb_{i}\dashrightarrow\pb'_{i,d}\cong\pb_{i+d}$ is locally given by $\theta_{i,d}^{\ast}(x'_{j,0})= x_{j,0}$ and $\theta_{i,d}^{\ast}(x'_{j,i+d})=f_j\,x_{j,i}$ over $U_j$, 
where $x'_{j,0}=x_{j,0}$ and $x'_{j,i+d}=x_{j,i+d}$ are local basis of $\calE_{i+d}$ over $U_j$. 
\end{Prop}
\begin{proof}
Since $S_{i,0}$ is defined by $x_{j,0}=0$ over $U_j$, $Y_{i,d}$ is defined by $x_{j,0}=0$ and $f_j\,x_{j,i}=0$ over $U_j$ for $j=0,1,2$. 
By the proof of \cite[(1.3) Lemma]{sumihiro}, $\pb'_{i,d}$ is the $\bP^1$-bundle associated with the vector bundle locally generated by $x'_{j,0}=x_{j,0}$ and $x'_{j,i+d}=f_j\, x_{j,i}$ on $U_j$. 
Since $f_j$ is a local basis of the ideal sheaf $\calO_{\bP^2}(-d)\subset\calO_{\bP^2}$ of $C_d$ over $U_j$, 
$x'_{j,0}$ and $x'_{j,i+d}$ form local basis of $\calO_{\bP^2}$ and $\calO_{\bP^2}(-i-d)$ on $U_j$, respectively. 
Hence $\pb'_{i,d}$ is isomorphic to $\pb_{i+d}$, and $\pbp'_{i,d}$ coincides with $\pbp_{i+d}$. 
Moreover, $x'_{j,0}=x_{j,0}$ and $x'_{j,i+d}=f_j\, x_{j,i}$ imply that $\theta_{i,d}^{\ast}(x'_{j,0})=x_{j,0}$ and $\theta_{i,d}^{\ast}(x'_{j,i+d})=f_j\,x_{j,i}$. 
\end{proof}

Let $X\subset\bP^3$ be a normal quartic surface with a node $P$ not containing any lines through $P$, 
and let $p:\bP^3\dashrightarrow \bP^2$ be the projection from $P$. 
Let $\phi:\widetilde{\bP}^3\to\bP^3$ be the blowing-up of $\bP^3$ at $P$, and 
let $\widetilde{X}$ be the strict transform of $X$ by $\phi$. 
The projection $p$ induces a morphism $\tilde{p}:\widetilde{\bP}^3\to\bP^2$. 
Note that $\widetilde{\bP}^3$ is isomorphic to $\pb_1$, and $\tilde{p}$ coincides with $\pbp_1:\pb_1\to\bP^2$. 
Let $\widetilde{X}\subset\pb_1$ be the strict transformation of $X$ by the blowing-up at $P$, and 
let $\pi_X:\widetilde{X}\to\bP^2$ denote the restriction of $\pbp_1$ to $\widetilde{X}$. 
Since $X$ contains no lines through $P$, $\pi_X$ is a double cover over $\bP^2$. 
Let $\Gamma_X\subset\bP^2$ denote the branch locus of $\pi_X$, 
and let $\Delta_X\subset\bP^2$ be the image of exceptional divisor of $\widetilde{X}\to X$ by $\pi_X$. 
\[\begin{diagram}
\node{X} \arrow{se,b,..}{p|_{X}} \node{\widetilde{X}} \arrow{w,t}{\phi|_{\widetilde{X}}} \arrow{e,t}{\mbox{\tiny emb.}} \arrow{s,l}{\pi_X} \node{\widetilde{\bP}^3} \arrow{sw,l}{\tilde{p}} \arrow{e,t}{\mbox{\tiny isom.}} \node{\pb_1} \arrow{wsw,b}{\pbp_1} \\
\node[2]{\bP^2}
\end{diagram}\]

\begin{Th}\label{thm. corr}
Let $P$ be a point of $\bP^3$. 
Let $X\subset\bP^3$ be a normal quartic surface  with a node $P$ not containing any lines through $P$.
Then the curve $\Delta_X$ is an even contact conic of $\Gamma_X$. 
Moreover, the following hold: 
\begin{enumerate}[\rm (i)]

\item\label{thm. corr 1} 
The following map is surjective: 
\[ \begin{array}{ccc}
\left\{ \substack{ \mbox{normal quartic surfaces in $\bP^3$ with a} \\[0.2em] 
\mbox{node $P$ which contain no lines through $P$} } \right\}  
& \to &
\left\{ \substack{ \mbox{pairs of sextic curves and} \\[0.2em] \mbox{their even contact conics}}  \right\} \\[1em]
X & \mapsto & (\Gamma_X,\Delta_X)
\end{array}\]

\item\label{thm. corr 2} 
Let $(x:y:z:w)$ be a system of homogeneous coordinates of $\bP^3$ such that 
$P=(0:0:0:1)$. 
Assume that the  defining equation of $X$ is given by $g_2\,w^2+2\,g_3\,w+g_4=0$.  
Then the defining equations of $\Gamma_X$ and $\Delta_X$ are  
\[ \Gamma_X: g_3^2-g_2\,g_4=0 \mbox{\ \  and \ } \Delta_X: g_2=0, \]  where $g_i\in\bC[x,y,z]$ is a homogeneous polynomial of degree $i$ for  $i=2,3,4$.

\item\label{thm. corr 3} 
If we further assume that $\Delta_X$ is a simple contact conic of $\Gamma_X$, there exist  natural surjective maps 
\[ \begin{array}{ccc}
\alpha_1:
\left\{ \substack{\mbox{hyperplanes in $\bP^3$} \\[0.2em] \mbox{not containing $P$} } \right\}  
& \to &
\left\{ \substack{ \mbox{cubic curves on $\bP^2$ passing through} \\[0.2em] \mbox{all tangent points of $\Delta_X$ and $\Gamma_X$,} \\[0.2em] \mbox{not containing $\Delta_X$} }  \right\} 
\end{array}\]
\[ \begin{array}{ccc}
\alpha_2:\left\{ \substack{\mbox{quadric surfaces in  $\bP^3$} \\[0.2em] \mbox{smooth at $P$.} } \right\}  
& \to &
\left\{ \substack{ \mbox{quartic curves on $\bP^2$ passing through} \\[0.2em] \mbox{all tangent points of $\Delta_X$ and $\Gamma_X$,} \\[0.2em] \mbox{not containing $\Delta_X$}}  \right\} 
\end{array}\]
satisfying the following conditions for each singular point $Q\in X\setminus\{P\}$; 
\begin{enumerate}[\rm (I)]
\item 
for a hyperplane $H_1\subset\bP^3$ with $P\not\in H_1$, $Q\in H_1$ if and only if $\pi_X(Q)\in\alpha_1(H_1)$; and 
\item
for a quadric surface $H_2\subset\bP^3$ smooth at $P\in H_2$, $Q\in H_2$ if and only if $\pi_X(Q)\in\alpha_2(H_2)$
\end{enumerate}
\end{enumerate}
\end{Th}
\begin{proof}
(\ref{thm. corr 1}) 
Let $X\subset\bP^3$ be a normal quartic surface with a node $P$ not containing any lines through $P$. 
Let $\phi:\widetilde{\bP}^3\to \bP^3$ be the blowing-up at $P$, 
and let $\widetilde{X}$ be the strict transform of $X$ under $\phi$. 
Note that there is an isomorphism $\widetilde{\bP}^3\cong\pb_1$ such that the morphism $\widetilde{\bP}^3\to\bP^2$ induced by the projection from $P$ coincides with $\pbp_1:\widetilde{\bP}^3\cong\pb_1\to\bP^2$,  and the exceptional divisor of $\phi$ coincides with the section $S_{1,0}\subset\pb_1$. 
Since $P$ is a node of $X$, $C_2:=\widetilde{X}\cap S_{1,0}$ is a smooth conic on $S_{1,0}$,  and so is $\Delta_X\subset\bP^2$. 
Since $X$ contains no lines through $P$, 
by the blowing-up along $C_2$ and Proposition~\ref{prop. birational}, 
we have the birational map $\theta_{1,2}:\pb_1\dashrightarrow \pb_3$, and $\theta_{1,2}$ induces an embedding $\iota:\widetilde{X}\to \pb_3\setminus S_{3,0}$ satisfying the following commutative diagram: 
\[\begin{diagram}
\node{\widetilde{X}} 
\arrow{e,t}{\iota} 
\arrow{se,b}{\pi_X}
\node{\pb_3\setminus S_{3,0}} \arrow{s,r}{\pbp_3} \\
\node[2]{\bP^2}
\end{diagram}\]
This implies that $\Gamma_X$ is a sextic curve, and $\Gamma_X$ is smooth at $\Gamma_X\cap \Delta_X$ since $\widetilde{X}$ is smooth over $\Delta_X$. 
Moreover, since $\pi_X^{\ast}\Delta_X$ is reducible, $\Delta_X$ must be an even contact conic of $\Gamma_X$. 

Conversely, let $\Gamma\subset\bP^2$ be a sextic curve with an even contact conic $\Delta$. 
Let $\pi_{\Gamma}:\widetilde{X}\to\bP^2$ be the double cover branched along $\Gamma$. 
There is an embedding $\iota:\widetilde{X}\to \bP_3\setminus S_{3,0}$ satisfying $\varphi_3\circ\iota=\pi_{\Gamma}$. 
Since $\Delta$ is rational, $\pi_{\Gamma}^{\ast}\Delta$ splits into two rational curves $\Delta_1$ and $\Delta_2$. 
By blowing-up $\pb_3$ along $\Delta_1$ and by Proposition~\ref{prop. birational}, we obtain 
an embedding $\iota':\widetilde{X}\to\pb_1$ such that $\iota'(\widetilde{X})\cap S_{1,0}=\pbp_1^{-1}(\Delta)\cap S_{1,0}$. 
Since $\widetilde{X}$ is smooth over $\Delta$, $X:=\phi\circ\iota'(\widetilde{X})\subset\bP^3$ is a normal quartic surface with a node $P$, which contains no lines through $P$, such that $(\Gamma_X,\Delta_X)=(\Gamma,\Delta)$. 

Assertion (\ref{thm. corr 2}) is clear. 
Hence we next prove assertion (\ref{thm. corr 3}). 
We may assume that $P=(0:0:0:1)$ and $X$ is defined by $g_2\, w^2+2\, g_3+g_4=0$ as in the assumption of (\ref{thm. corr 2}). 
By the argument of the proof of (\ref{thm. corr 1}), we have the birational map $\tilde{\theta}=\theta_{1,2}\circ\phi^{-1}:\bP^3\dashrightarrow\pb_3$ over $\bP^2$, 
which is locally given on $\{(x:y:z:w)\in\bP^3\mid z=1\}$ by 
\[ \left({x}\,,{y},\,{w}\right) \mapsto \left(x,\,y,\,t\right)=\left({x},\,{y},\,{w}\,g_2({x},{y},1)\right), \]
where we regard $(x:y:z)$ as a system of homogeneous coordinates of $\bP^2$, and 
$t$ is a local coordinate of a fiber of $\pb_3$ such that $t=g_2=0$ gives the center of the blowing-up $\sigma'_{3,2}:V_{3,2}\to\pb_3$ in the diagram (\ref{diagram}). 
Then $\tilde{\theta}$ gives a birational map from $X\subset\bP^3$ to the subvariety $\widetilde{X}\subset\pb_3$ locally given by 
\[ (t+g_3)^2-g_3^2+g_2\,g_4=0, \]
which is naturally the double cover of $\bP^2$ branched along $\Gamma_X$. 
Note that $t+g_3=0$ gives a section $\widetilde{S}$ of $\pb_3$ which contains all singularities of $\widetilde{X}$. 

Next we construct the map $\alpha_1$. 
For a hyperplane $H\subset\bP^3$ given by $w+\ell=0$ for some $\ell\in\bC[x,y,z]$, 
$\tilde{\theta}$ gives a birational map from $H$ to a section $\widetilde{L}$ of $\pb_3$ locally given by 
\[ t+g_2\,\ell=0. \]
We define the map $\alpha_1$ by $\alpha_1(H)=\pbp_3(\widetilde{S}\cap\widetilde{L})$. 
Since $\alpha_1(H)$ is given by $g_2\,\ell-g_3=0$, 
$\alpha_1(H)$ passes through all tangent points of $\Delta_X$ and $\Gamma_X$, and $\alpha_1(H)$ does not contain $\Delta_X$. 
Since $\widetilde{S}$ contains all singularities of $\widetilde{X}$, 
$H$ contain a singularity $Q\in X\setminus\{P\}$ if and only if $p(Q)\in\alpha_1(H)$. 

To construct the map $\alpha_2$, let $H$ be a quadratic surface of $\bP^3$ given by 
\[ a_1\,w+a_2=0, \]
where $a_i\in\bC[x,y,z]$ is a homogeneous polynomial of degree $i$ for each $i=1,2$ with $a_1\ne 0$. 
The birational map $\tilde{\theta}$ gives a birational map from $H$ to a closed subset $\widetilde{H}$ of $\pb_3$ locally given by 
\[ a_1\,t+a_2\,g_2=0. \]
We put $\alpha_2(H)$ as $\pbp_3(\widetilde{S}\cap\widetilde{H})$. 
Since $\alpha_2(H)$ is defined by $a_1\,g_3-a_2\,g_2=0$, $\alpha_2(H)$ passes through all tangent points of $\Delta_X$ and $\Gamma_X$, and $\alpha_2(H)$ does not contain $\Delta_X$. 
Moreover, $H$ contains a singular point $Q\in X\setminus\{P\}$ if and only if $p(Q)\in\alpha_2(H)$. 

We finally prove that $\alpha_1$ and $\alpha_2$ are surjective. 
Let $C_3$ and $C_4$ be curves on $\bP^2$ of degree $3$ and $4$, respectively,  
through all tangent points of $\Delta_X$ and $\Gamma_X$, 
which do not contain $\Delta_X$. 
From the exact sequence 
\[ 0\to H^0(\bP^2,\calO_{\bP^2}(i-2))\overset{\times g_2}{\to} H^0(\bP^2,\calO_{\bP^2}(i))\to H^0(\Delta_X,\calO_{\Delta_X}(i))\to 0 \]
for $i=3,4$, we may assume that $C_3$ and $C_4$ are defined by 
\[ \ell_0\,g_3-\ell_1\,g_2=0 \ \mbox{ and } \ a_1\,g_3-a_2\,g_2=0, \]
respectively, where $\ell_0\in\bC$, $\ell_1,a_1\in H^0(\bP^2,\calO_{\bP^2}(1))$ and $a_2\in H^0(\bP^2,\calO_{\bP^2}(2))$. 
Since $C_3$ and $C_4$ do not contain $\Delta_X$, $\ell_0\ne0$ and $a_1\ne0 $. 
Hence $C_3=\alpha(H_1)$ and $C_4=\beta(H_2)$, 
where $H_1$ and $H_2$ are given by 
$\ell_0\,w+\ell_1=0$ and $a_1\,w+a_2=0$, respectively. 
\end{proof}

\section{Nodal quartic surfaces in $\bP^3$}\label{sec. quartic surfaces}

In this section, we recall a geometric difference of syzygetic and asyzygetic quartic surfaces in detail. 
In order to do this, we investigate the dimension of the linear systems on $\bP^3$ consisting of quadric surfaces through given points. 

\subsection{Linear systems on $\bP^1\times\bP^1$ and $\bP^3$}

In this subsection, we give a bound on the dimension of linear systems on $\bP^3$ with $8$ assigned base points in ``general" position. 
We first consider linear systems on $\bP^1\times\bP^1$. 
Let $p_1$ and $p_2$ denote the two projections from $\bP^1\times\bP^1$ to $\bP^1$. 

\begin{Lem}\label{lem. 1,1}
Let $P_1,\dots,P_r$ be $r$ distinct points of $\bP^1\times\bP^1$, and let $C$ be a $(1,1)$-curve on $\bP^1\times\bP^1$. 
Put $\fd:=|C-P_1-\dots-P_r|$. 
\begin{enumerate}[\rm (i)]
\item\label{lem. 1,1 1}
In the case where $r=2$, $\fd$ has no unassigned base points if 
$p_j(P_1)\ne p_j(P_2)$ for  $j=1,2$.
\item\label{lem. 1,1 2} 
In the case where $r=3$, 
$\dim\fd=0$ if
$\sharp(p_j^{-1}(Q)\cap\{P_1,P_2,P_3\})\leq2$ for any $Q\in\bP^1$ and  $j=1,2$. 
\end{enumerate}
\end{Lem}
\begin{proof}
(\ref{lem. 1,1 1}) 
The divisors $p_1^{\ast}(p_1(P_1))+p_2^{\ast}(p_2(P_2))$ and $p_1^{\ast}(p_1(P_2))+p_2^{\ast}(p_2(P_1))$ are elements of $\fd$, and they meet transversally at $P_1$ and $P_2$. 
This implies that $\fd$ has no unassigned base points. 

(\ref{lem. 1,1 2}) 
By the assumption, we may assume that  $p_j(P_1)\ne p_j(P_2)$ for any $j=1,2$. 
Since $\dim |C|=3$, (\ref{lem. 1,1 1}) implies that $\dim\fd=0$. 
\end{proof}

\begin{Prop}\label{prop. general}
Let $P_1,\dots,P_r$ be $r$ points of $\bP^1\times\bP^1$, 
and let $D$ be a $(2,2)$-curve. 
Assume that 
\begin{enumerate}[{\rm (a)}]
\item\label{prop. general 1}
for any $Q\in\bP^1$ and any $j=1,2$, $\sharp(p_j^{-1}(Q)\cap\{P_1,\dots,P_r\})\leq2$, and 
\item\label{prop. general 2} 
no $(1,1)$-curves contain five points of $P_1,\dots,P_r$. 
\end{enumerate}
Then, if $r=7$, $\dim |D-P_1-\dots-P_7|= 1$. 
In particular, if $r=8$, then $0\leq\dim |D-P_1-\dots-P_8|\leq 1$. 
\end{Prop}
\begin{proof}
We put $\fd:=|D-P_1-\dots-P_6|$. 
Since $\dim |D|=8$, it is sufficient to prove that $P_7$ is not an unassigned base point of $\fd$. 
We put $F_{j,i}:=p_j^{\ast}(p_j(P_i))$ for $i=1,\dots,7$ and $j=1,2$. 

\textit{Case 1.} 
Suppose that there exists a $(1,1)$-curve $C$ passing through four points of $P_1,\dots,P_6$. 
We may assume that $C_{1234}:=C$ passes through $P_1,\dots, P_4$. 
By the assumption (\ref{prop. general 2}), $P_7\not\in C_{1234}$. 
If $F_{j,5}\ne F_{j,6}$ for any $j=1,2$, then there exists a $(1,1)$-curve $C_{56}$  passing through  $P_5$ and $P_6$ but not $P_7$ by Lemma~\ref{lem. 1,1}. 
Hence $C_{1234}+C_{56}$ is a $(2,2)$-curve not passing through $P_7$, but $P_1,\dots,P_6$. 
If $F_{j,5}=F_{j,6}$ for some $j=1,2$, say $j=1$, then $P_7\not\in F_{1,5}$ by the assumption (\ref{prop. general 1}). 
By choosing $Q\in\bP^1$ such that $p_2(P_7)\ne Q$, 
$C_{1234}+F_{1,5}+p_2^{\ast}(Q)$ is a $(2,2)$-curve not passing through $P_7$, but $P_1,\dots,P_6$. 

\textit{Case 2.} 
Suppose that no $(1,1)$-curves pass through four points of $P_1,\dots,P_6$. 
Let $C_{ijk}$ be the $(1,1)$-curve passing through distinct three points $P_i,P_j,P_k$ of $P_1,\dots,P_6$. 
Note that $C_{ijk}$ is determined uniquely by Lemma~\ref{lem. 1,1} (\ref{lem. 1,1 2}), 
and $C_{i j k}\ne C_{i'j'k'}$ if $\{i,j,k\}\ne\{i',j',k' \}$. 
If $P_7\not\in C_{123}$ and $P_7\not\in C_{456}$, we have nothing to prove. 

Suppose $P_7\in C_{123}$, and $C_{123}$ is reducible. 
By  assumption (\ref{prop. general 1}), we may assume that $C_{123}=F_{1,1}+F_{2,3}$, $F_{1,1}=F_{1,2}$ and $F_{2,3}=F_{2,7}$. 
Also by assumption (\ref{prop. general 1}), $F_{1,1}$ and $F_{2,3}$ are not irreducible components of $C_{13i}$ for any $i=4,5,6$. 
Since the intersection number $C_{123}.C_{13i}=2$, $C_{13i}$ does not pass through $P_7$ for $i=4,5,6$. 
If $P_7\in C_{2ij}$ for any $i,j=4,5,6$ with $i\ne j$, then it follows from $C_{245}.C_{246}=2$ and  assumption (\ref{prop. general 1}) that $C_{245}$ and $C_{246}$ are reducible and have one common component; 
in this case, the common component must contain three points $P_2, P_4, P_7$, 
which is a contradiction to (\ref{prop. general 1}). 
Thus $P_7\not\in C_{2ij}$ for some $i,j=4,5,6$, say $i=4,j=5$, 
and $C_{136}+C_{245}$ is a $(2,2)$-curve not passing through $P_7$, but $P_1,\dots,P_6$. 

Suppose $P_7\in C_{123}$, and $C_{123}$ is irreducible. 
Then $P_7\not\in C_{12i}$ for $i=4,5,6$. 
Suppose $P_7\in C_{3ij}$ for some $i,j=4,5,6$, say $i=4, j=5$, then we may assume by the above argument that $C_{345}$ is irreducible. 
In this case, we have $P_7\not\in C_{346}$. 
Thus $C_{125}+C_{346}$ is a $(2,2)$-curve not passing through $P_7$, but $P_1,\dots,P_6$. 
Therefore we have proved the assertion. 
\end{proof}

\begin{Def}{\rm
For $r$ points $P_1,\dots,P_r\in\bP^1\times\bP^1$ with $r\leq 8$, 
we say that $P_1,\dots,P_r$ are \textit{in general position} if 
$P_1,\dots,P_r$ satisfy the assumption (\ref{prop. general 1}) and (\ref{prop. general 2}) in Proposition~\ref{prop. general}. 
}\end{Def}

We will use the following lemma later. 

\begin{Lem}\label{lem.1,2}
Let $P_1,\dots,P_5\in\bP^1\times\bP^1$ be five points in general position, 
and let $D$ be a $(1,2)$-curve on $\bP^1\times\bP^1$. 
Then $\dim\,|D-P_1-\dots-P_5|=0$. 
\end{Lem}
\begin{proof}
We have $\dim\,|D-P_1-\dots-P_5|\geq0$ since $\dim\,|D|=5$. 
It is sufficient to prove that $P_5$ is not an unassigned base point of $\fd:=|D-P_1-\dots-P_4|$. 
From the condition (\ref{prop. general 1}) in Proposition~\ref{prop. general}, there are $i_1,i_2\in\{1,\dots,4\}$ such that 
$p_2(P_{i_1})\ne p_2(P_5)$ and $p_2(P_{i_2})\ne p_2(P_5)$, say $i_1=1$ and $i_2=2$. 
Let $C_{ijk}$ be the $(1,1)$-curve as in the proof of Proposition~\ref{prop. general}. 
If $P_5\in C_{i34}$ for $i=1,2$, then $C_{134}$ and $C_{234}$ have a common component $F$ which contains $P_5$. 
This implies that $P_5$ is contained in a ruling of $\bP^1\times\bP^1$ passing through two points of $P_1,\dots,P_4$, which is a contradiction. 
Hence $P_5\not\in C_{i34}$ for some $i=1,2$, say $i=2$. 
Then $D_{234}+p_2^{\ast}(p_2(P_1))$ is an element of $\fd$ no passing through $P_5$. 
\end{proof}

Next we consider linear systems on $\bP^3$. 
We say that $r$ points $P_1,\dots,P_r\in\bP^3$ with $r\leq 8$ are \textit{in general position} 
if 
\begin{enumerate}[\rm (A)]
\item\label{def. General 1} 
no three points of $P_1,\dots,P_r$ are collinear, and
\item\label{def. General 2} 
no hyperplanes in $\bP^3$ contain five points of $P_1,\dots,P_r$. 
\end{enumerate}

\begin{Prop}\label{prop. General2}
Let $P_1,\dots,P_8\in\bP^3$ be eight points in general position, and 
let $D$ be a quadratic surface in $\bP^3$, i.e., $\deg D=2$. 
Then $1\leq\dim |D-P_1-\dots-P_8|\leq 2$. 
\end{Prop}
\begin{proof}
We put $\fd:=|D-P_1-\dots-P_8|$. 
Since $\dim |D|=9$, we have $\dim \fd\geq1$. 
The condition (\ref{def. General 1}) and (\ref{def. General 2}) implies that the number of reducible members in $\fd$ is finite. 
Hence there exist irreducible members in $\fd$. 

First we consider the case where there is a smooth quadratic surface $D_0$ in $\fd$. 
The conditions (\ref{def. General 1}) and (\ref{def. General 2}) imply that $P_1,\dots,P_8$ are in general position 
as points on $D_0\cong\bP^1\times\bP^1$. 
Hence $0\leq\dim\,\fd|_{D_0}\leq1$ from Proposition~\ref{prop. general}, and we have $1\leq\dim\,\fd\leq2$.

If all members of $\fd$ are singular, then there exists a point $P'\in\bP^3$ such that $P'$ is a singular point of all members of $\fd$. 
Let $\pi_{P'}:\bP^3\setminus\{P'\}\to\bP^2$ be the projection from $P'$. 
In this case, for each member $D'\in\fd$, there exists a conic $C'\subset\bP^2$ with $\pi_{P'}(P_i)\in C'$ ($i=1,\dots,8$) such that $D'$ is the closure of $\pi_{P'}^{-1}(C')$ in $\bP^3$. 
From the conditions (\ref{def. General 1}), (\ref{def. General 2}) and $\dim \fd\geq1$, 
the image of $\{P_1,\dots,P_8\}$ under $\pi_{P'}$ is a set of $4$ points 
such that no three points of them are collinear. 
This implies that $\dim\fd=1$. 
\end{proof}

\begin{Prop}\label{prop. dim3}
Let $P_1,\dots,P_8\in\bP^3$ be eight points in general position, 
and $D$ a quadratic surface in $\bP^3$. 
Put $\fd:=|D-P_1-\dots-P_8|$. 
If $\dim\,\fd=2$, then $\fd$ has no unassigned base points. 
\end{Prop}
\begin{proof}
From the proof of Proposition~\ref{prop. General2}, there is a smooth member $D_0$ of $\fd$. 
Since $\dim\,\fd|_{D_0}=1$, Lemma~\ref{lem. 1,1} and \ref{lem.1,2} imply that $\fd|_{D_0}$ has no fixed components. 
Thus the set of base points of $\fd|_{D_0}$ is $\{P_1,\dots,P_8\}$ as points of $D_0 \cong \bP^1\times\bP^1$. 
Hence the set of base points of $\fd$ is $\{P_1,\dots,P_8\}$. 
Moreover, the intersection number $(D_0,D_1,D_2)=8$ for $D_i\in \fd$ ($i=0,1,2$) which span $\fd$. 
This implies that $D_0$, $D_1$ and $D_2$ meet at $P_1,\dots,P_8$ transversally. 
Thus $\fd$ has no unassigned base points. 
\end{proof}

\begin{Rem}\label{rem. base pts}\rm
Let $D$ be a quadratic surface in $\bP^3$, and 
let $P_1,\dots,P_8$ be eight points in $\bP^3$. 
Put $\fd:=|D-P_1-\dots-P_8|$. 
If $\fd$ has no unassigned base points, then $P_1,\dots,P_8$ are in general position.  
\end{Rem}
\begin{proof}
Suppose that there exists a line $L\subset\bP^3$ with $P_1,P_2,P_3\in L$. 
Let $D_0$ be a member of $\fd$, and let $H$ be a hyperplane with $L\subset H$ and $H\not\subset D_0$. 
Then $D_0|_H$ is a conic on $H\cong\bP^2$ passing through $P_1,P_2,P_3$. 
Thus $L\subset D_0$, which is a contradiction. 

Suppose there exists a hyperplane $H\subset\bP^3$ with $P_1,\dots,P_5\in H$. 
By the above argument, we may assume that no three points of $P_1,\dots,P_5$ are collinear. 
If $D_0\in\fd$, then $D_0|_H$ is the unique conic on $H$ passing through $P_1,\dots,P_5$, 
which is a contradiction 
\end{proof}

We prove the following lemma about a linear system on the Hirzebruch surface $p:\Sigma_2\to\bP^1$ of degree two. 
Let $S_0$ be the section of $p$ with $S_0^2=-2$. 
Note that there is a morphism $\sigma:\Sigma_2\to\bP^3$ whose image is a cone of degree two in $\bP^3$ and which is the contraction of $S_0$. 

\begin{Lem}\label{lem. hirzebruch}
Let $p:\Sigma_2\to\bP^1$ be the Hirzebruch surface of degree two, and 
$\sigma:\Sigma_2\to\bP^3$ as above. 
Let $S_0$ and $F$ denote the section with $S_0^2=-2$ and a fiber of $p$, respectively. 
If an effective divisor $D$ on $\Sigma_2$ is linearly equivalent to $S_0+2\,F$, 
then there is a hyperplane $H$ of $\bP^3$ satisfying $\sigma^{\ast}H=D$. 
\end{Lem}
\begin{proof}
We have a morphism $\sigma^{\ast}:|H|\to |S_0+2\,F|$. 
It is clear that $\sigma^{\ast}$ is injective. 
Since $\dim|H|=\dim|S_0+2\,F|=3$, $\sigma^{\ast}$ is bijection. 
\end{proof}

\subsection{Syzygetic quartic surfaces}

In this subsection we recall the geometry of nodes on a quartic surface.

\begin{Def}\label{def. conical}{\rm
Let $X\subset\bP^3$ be a hypersurface. 
\begin{enumerate}[\rm (i)]
\item We call $X$ a \textit{nodal surface} if all singularities of $X$ are nodes, i.e. $A_1$-singularities. 
Moreover, for a positive integer $r$, we call a nodal surface $X$ a \textit{$r$-nodal surface} if the number of nodes of $X$ is equal to $r$.  
\item 
We call a nodal quartic surface $X$ a \textit{syzygetic quartic surface} if 
there are three quadratic forms $f_1,f_2,f_3\in H^0(\bP^3,\calO_{\bP^3}(2))$ 
such that $X$ is defined by $f_1^2+f_2^2+f_3^2=0$, 
otherwise, we call $X$ an \textit{asyzygetic quartic surface}. 
\item
Let $X$ be a syzygetic quartic surface defined by $f_1^2+f_2^2+f_3^2=0$. 
We call a node $P\in X$ an \textit{assigned node} with respect to $f_1,f_2,f_3$ if $f_1=f_2=f_3=0$ at $P$, 
otherwise, we call $P$ an \textit{unassigned node}. 
\end{enumerate}
}\end{Def}

\begin{Rem}\rm
By uniqueness of canonical forms of quadratic forms, 
assigned nodes of a syzygetic quartic surface do not depend on the choice of quadratic forms $f_1, f_2, f_3\in H^0(\bP^3,\calO_{\bP^3}(2))$. 
Hence we simply call an assigned nodes with respect to $f_1, f_2, f_3$ an assigned nodes. 
\end{Rem}

\begin{Prop}\label{prop. conical}
Let $X\subset\bP^3$ be a nodal quartic surface. 
Then $X$ is syzygetic if and only if there are eight nodes $P_1,\dots,P_8\in X$ in general position with $\dim\fd=2$, where $\fd:=|D-P_1-\dots-P_8|$ for a quadratic surface $D\subset\bP^3$. 
In this case, the eight points $P_1,\dots,P_8$ are the assigned nodes of $X$. 
\end{Prop}
\begin{proof}
Suppose that $X$ is defined by $f_1^2+f_2^2+f_3^2=0$ for some quadratic forms $f_1,\, f_2,\, f_3$. 
Let $C_i$ be the quadratic surface given by $f_i=0$ for $i=1,2,3$. 
Then $X$ has singularities at $C_1\cap C_2\cap C_3$. 
Hence the set $C_1\cap C_2\cap C_3$ is contained in the singular locus of $X$. 
This implies that $f_1,f_2$ and $f_3$ are linearly independent over $\bC$ 
since the singular locus of $X$ consists of finite points. 
Moreover, since the intersection number $(C_1,C_2,C_3)=8$, 
if $\sharp (C_1\cap C_2\cap C_3)<8$, then $X$ has a singularity which is not a node; 
this is a contradiction. 
Hence $C_1\cap C_2\cap C_3$ consists of eight points $P_1,\dots,P_8$, 
and $C_1,C_2,C_3\in\fd:=|D-P_1-\dots-P_8|$, and $\dim \fd\geq 2$.  
If there exists a line $L\subset\bP^3$ through three points of $P_1,\dots,P_8$, 
then $L\subset C_i$ ($i=1,2,3$) since the intersection number $(L, C_i)=2$; 
hence $X$ is singular along $L$, which is a contradiction. 
Thus no three of the eight nodes of $X$ are collinear. 
If there exists a hyperplane $H\subset\bP^3$ through five points of $P_1,\dots,P_8$, 
then either $H\subset C_i$ or $C_i|_H$ is the unique conic on $H$ through the five nodes of $X$; 
hence $\dim(C_1\cap C_2\cap C_3)\geq 1$, 
which is a contradiction. 
Therefore $P_1,\dots,P_8$ are in general position.  
From Proposition~\ref{prop. General2}, we have $\dim\fd=2$. 

Suppose that there exist eight nodes $P_1,\dots,P_8\in X$ in general position and  $\dim\fd=2$. 
Let $\Phi_{\fd}:\bP^3\dashrightarrow\bP^2$ be the rational map defined by $\fd$, and let
$\sigma:\widetilde{\bP}^3\to\bP^3$ be the blowing-up of $\bP^3$ at the eight points $P_1,\dots,P_8$. 
By Proposition~\ref{prop. dim3}, we have the morphism $\widetilde{\Phi}_{\fd}:\widetilde{\bP}^3\to\bP^2$ such that $\widetilde{\Phi}_{\fd}=\Phi_{\fd}\circ\sigma$. 
Moreover, the proper transformation $\widetilde{X}$ of $X$ under $\sigma$ gives the minimal resolution of singularities $P_1,\dots,P_8$ of $X$ since $P_1,\dots,P_8$ are $A_1$-singularities of $X$. 
We denote the inclusion from $\widetilde{X}$ to $\widetilde{\bP}^3$ by $\iota:\widetilde{X}\hookrightarrow\widetilde{\bP}^3$, and put $\pi:=\widetilde{\Phi}_{\fd}\circ\iota:\widetilde{X}\to\bP^2$. 
\[\begin{diagram}
\node{\widetilde{X}} \arrow{e,t}{\iota} \arrow{se,b}{\pi} \node{\widetilde{\bP^3}} \arrow{s,r}{\widetilde{\Phi}_{\fd}} \arrow{e,t}{\sigma} \node{\bP^3} \arrow{sw,b,..}{\Phi_{\fd}} \\
\node[2]{\bP^2}
\end{diagram} \] 
Let $E_1,\dots,E_8$ be the exceptional divisors of $\sigma$. 
Since $\widetilde{\Phi}_{\fd}^{\ast}\calO_{\bP^2}(1)=\calO_{\widetilde{\bP}^3}(\sigma^{\ast}D-E_1-\dots-E_8)$, for $i=1,\dots,8$, we have 
\[
\widetilde{\Phi}_{\fd}^{\ast}\calO_{\bP^2}(1)\otimes\calO_{E_i} \cong \calO_{\widetilde{\bP}^3}(-E_i)\otimes\calO_{E_i} \cong \calO_{E_i}(1). 
\]
Hence the restriction of $\widetilde{\Phi}_{\fd}$ to $E_i$ is an isomorphism  $\widetilde{\Phi}_{\fd}|_{E_i}:E_i\overset{\sim}{\to}\bP^2$. 
Since $P_i$, $i=1,\dots,8$, are double points, $\Gamma_i:=\widetilde{\Phi}_{\fd}(E_i\cap \widetilde{X})$ are irreducible conics in $\bP^2$. 
For a line $L$ on $\bP^2$, since $\pi^{\ast}L\sim\iota^{\ast}(\sigma^{\ast}D-E_1-\dots-E_8)$, 
the self-intersection number $(\pi^{\ast}L)^2=0$. 
This implies $\dim\pi(\widetilde{X})\leq1$. 
Since $\Gamma_i\subset\pi(\widetilde{X})$, $\Gamma:=\pi(\widetilde{X})$ is an irreducible conic in $\bP^2$. 
We may assume that $\Gamma$ is given by $u_1^2+u_2^2+u_3^2=0$, where $u_1,u_2,u_3$ are homogeneous coordinates of $\bP^2$. 
Let $X'$ be the quartic surface given by $f_1^2+f_2^2+f_3^2=0$, where $f_i=\Phi_{\fd}^{\ast}u_i$. 
Note that $X'$ is the closure of $\Phi^{-1}(\Gamma)$ in $\bP^3$. 
Hence $X\subset X'$, and since both surfaces $X$ and $X'$ are quartic surfaces, we have $X=X'$. Therefore $X$ is syzygetic. 
Moreover, it is clear that $P_1,\dots,P_8$ are the assigned nodes of $X$. 
\end{proof}

\section{Improvement of criterion for nodal splitting curves of type $(2,4)$}\label{sec. improve}

In this section, we show a correspondence of syzygetic quartic surfaces and nodal splitting curves of type $(2,4)$. 
Let $X\subset\bP^3$ be a nodal quartic surface with a node at $\overline{P}_0$ 
not containing any lines through $\overline{P}_0$. 
Let $p:X\setminus\{\overline{P}_0\}\to\bP^2$ be the projection from $\overline{P}_0$, 
$\Gamma_X\subset\bP^2$ the branch locus of the double cover over induced by $p$, 
and $\Delta_X$ the image of exceptional divisor of the blowing-up of $X$ at $\overline{P}_0$ (see Section~5 for detail). 
Note that, for a sextic curve $\Gamma$ with an even contact conic $\Delta$, 
there exists a quartic surface $X\subset\bP^3$ such that $\Gamma_X=\Gamma$ and $\Delta_X=\Delta$ by Theorem~\ref{thm. corr}. 
Throughout this following, we assume that 
$X\subset\bP^3$ satisfies the above conditions and furthermore 
assume that $\Delta_X$ is a simple contact conic of $\Gamma_X$. 
Let $T_1,\dots,T_6$ denote the distinct tangent points of $\Gamma_X$ and $\Delta_X$.

\begin{Prop}\label{prop. 3,3 split}
Let $X\subset\bP^3$, $\overline{P}_0$ and $p:X\setminus\{\overline{P}_0\}\to\bP^2$ be as above. 
If there exist six nodes $\overline{P}_1,\dots,\overline{P}_6\in X$ on a conic $\overline{C}_0$ in a hyperplane $\overline{H}\subset\bP^3$, and no four points of $\overline{P}_1,\dots,\overline{P}_6$ are collinear, 
then $\Gamma_X$ is a splitting curve of type $(3,3)$ with respect to $\Delta_X$. 
Moreover, there is a $(3,3)$-curve $D^+\subset\bP^1\times\bP^1$ such that 
\[ \pi_{\Delta_X}^{\ast}\Gamma_X=D^++D^- \mbox{ and } \pi_{\Delta_X}(D^+\cap D^-)=\{P_1,\dots,P_6, T_1,\dots, T_6\}, \]
where $D^-=\iota_{\Delta_X}^{\ast}D^+$ and $P_i=p(\overline{P}_i)$ for $i=1,\dots,6$. 
\end{Prop}

\begin{proof}
Without loss of generality, we may assume that
\[ (\star) \begin{cases}
\overline{P}_0=[0:0:0:1], \\
\text{$X$ is defined by $g_2 w^2+2 g_3 w+g_4=0$, }
\end{cases}\]
where $g_i\in\bC[x,y,z]$ ($i=2,3,4$) are homogeneous polynomials of degree $i$. 
We may also assume that the hyperplane $\overline{H}$ is defined by $w=0$. 
Note that $X\cap \overline{H}$ is a quartic divisor on $\overline{H}$ which is singular at the six points $\overline{P}_1,\dots,\overline{P}_6$. 
Since no four points of $\overline{P}_1,\dots,\overline{P}_6$ are collinear, 
$\overline{C}_0$ is the unique conic on $\overline{H}$ through $\overline{P}_1,\dots,\overline{P}_6$. 
Hence $w=g_4=0$ defines the divisor $2\overline{C}_0$ on $\overline{H}$. 
Thus there is a homogeneous polynomial $f_2\in\bC[x,y,z]$ of degree $2$ such that 
$g_4=f_2^2$, and 
$\{\overline{P}_1,\dots,\overline{P}_6\}$ is defined by $w=g_3=f_2=0$. 
Hence $g_3=0$ defines a cubic curve on $\bP^2$ through $P_1,\dots,P_6, T_1,\dots,T_6$. 
By Theorem~\ref{thm. split}, $\Gamma_X$ is a splitting curve of type $(3,3)$ with respect to $\Delta_X$. 
The second assertion follows from the proof of Theorem~\ref{thm. split}. 
\end{proof}

\begin{Lem}\label{lem. 3,3 and 2,4}
Let $\Gamma\subset\bP^2$ be a nodal sextic curve with a simple contact conic $\Delta$. 
Assume that $\pi_{\Delta}^{\ast}\Gamma=D^++D^-$ for a $(2,4)$-curve $D^+$ and $D^-=\iota_{\Delta}^{\ast}D^+$. 
Let $P_1,\dots,P_7$ be the seven nodes of $\Gamma$ such that $\pi_{\Delta}(D^+\cap D^-)=\{P_1,\dots,P_7,T_1,\dots,T_6\}$. 
Then the following conditions hold; 
\begin{enumerate}[\rm (i)]
\item\label{lem. 3,3 and 2,4 1}%
no four points of $P_1,\dots,P_7$ are collinear; 
\item\label{lem. 3,3 and 2,4 2}%
if $\pi_{\Delta}^{\ast}\Gamma=E^++E^-$ for a $(3,3)$-curve $E^+\subset\bP^1\times\bP^1$ and $E^-=\iota_{\Delta}^{\ast}E^+$, then 
\[ \pi_{\Delta}(E^+\cap E^-)\not\subset\pi_{\Delta}(D^+\cap D^-)=\{P_1,\dots,P_7,T_1,\dots,T_6\}. \]
\end{enumerate}
\end{Lem}
\begin{proof}
We first prove the condition (\ref{lem. 3,3 and 2,4 1}). 
Suppose that four nodes $P_1,\dots,P_4\in \Gamma$ are on a line $L_0\subset\bP^2$. 
Since $L_0.\Gamma_X=6$, $L_0$ is a component of $\Gamma_X$, say $\Gamma_X=L_0+\Gamma_X'$. 
Hence $D^+=L_0^++D_0^+$ and $D^-=L_0^-+D_0^-$, where $\pi_{\Delta_X}^{\ast}L_0=L_0^++L_0^-$ and $\pi_{\Delta_X}^{\ast}\Gamma_X'=D_0^++D_0^-$. 
Note that $D_0^-$ and $L_0^+$ are either $(3,2)$- and $(0,1)$-curves or $(4,1)$- and $(1,0)$-curves respectively. 
In this case, the set $\{P_1,\dots,P_4\}$ is the image of $L_0^+\cap D_0^-$ under $\pi_{\Delta}$, 
which is contradiction to $L_0^+.D_0^-\leq 3$. 
Therefore no four nodes of $P_1,\dots,P_7$ are collinear. 

Suppose that $\pi_{\Delta}^{\ast}\Gamma=E^++E^-$ for a $(3,3)$-curve $E^+\subset\bP^1\times\bP^1$ and $E^-=\iota_{\Delta}^{\ast}E^+$. 
Let $C^+$ be the common component of $D^+$ and $E^+$, and put $D^+=D_0^++C^+$, $E^+=E^+_0+C^+$,  $D^-=D_0^-+C^-$ and $E^-=E_0^-+C^-$, 
where $C^-=\iota_{\Delta}^{\ast}C^+$. 
Then we have $E_0^+=D_0^-$ and $E_0^-=D_0^+$. 
Since $E_0^+\cap C^-\subset E^+\cap E^-$ and $E_0^+\cap C^-=D_0^-\cap C^-\not\subset D^+\cap D^-$, 
we obtain $E^+\cap E^-\not\subset D^+\cap D^-$. 
Thus the condition (\ref{lem. 3,3 and 2,4 2}) holds. 
\end{proof}

\begin{Th}\label{thm. conical}
Let $X\subset\bP^3$, $\overline{P}_0$ and $p:X\setminus\{\overline{P}_0\}\to\bP^2$ be as in Proposition~\ref{prop. 3,3 split}. 
Then the followings are equivalent: 
\begin{enumerate}[\rm (i)]
\item\label{thm. conical 1} 
$X$ is syzygetic such that $\overline{P}_0$ is an assigned node of $X$; 
\item\label{thm. conical 2} 
$\Gamma_X$ is a splitting curve of type $(2,4)$ with respect to $\Delta_X$; 
\item\label{thm. conical 3}
there exist seven nodes $\overline{P}_1,\dots,\overline{P}_7$ of $X$ satisfying the following conditions: 
\begin{enumerate}[\rm ({iii-}a)]
\item\label{thm. conical 3-1} 
there exist no conics on $\bP^2$ through the seven points $P_1,\dots, P_7$, 
where $P_i=p(\overline{P}_i)$ for $i=1,\dots,7$; 
\item\label{thm. conical 3-2}
$\dim\fd\geq 2$, 
where $\fd:=|4\,L-P_1-\dots-P_7-T_1-\dots-T_6|$ for a line $L\subset\bP^2$; 
\item\label{thm. conical 3-3} 
if three points $P_{i_1}, P_{i_2}, P_{i_3}$ of $P_1,\dots, P_7$ are collinear, 
then any cubic curve through the nine points  $P_{i_1},P_{i_2},P_{i_3},T_1,\dots,T_6$ contains $\Delta_X$; 
\item\label{thm. conical 3-4}
 for any five points $P_{i_1},\dots,P_{i_5}$ of $P_1,\dots,P_7$, 
there exist no cubic curves through the eleven points $P_{i_1},\dots,P_{i_5},T_1,\dots,T_6$. 
\end{enumerate}
\end{enumerate}
\end{Th}
\begin{proof}
We first prove that  conditions (\ref{thm. conical 1}) and (\ref{thm. conical 3}) are equivalent. 
Suppose that  condition (\ref{thm. conical 1}) holds. 
Let $\overline{P}_1,\dots,\overline{P}_7$ be the other assigned nodes of $X$. 
By  Proposition~\ref{prop. conical} and Theorem~\ref{thm. corr} (\ref{thm. corr 3}),  condition (\ref{thm. conical 3}-b) holds. 
Since the eight points $\overline{P}_0,\dots,\overline{P}_7$ are in general position, 
we have $\dim\fd=2$. 
Let $\bar{\fd}$ be the linear system $|2\overline{H}-\overline{P}_0-\dots-\overline{P}_7|$ on $\bP^3$, where $\overline{H}$ is a hyperplane. 
Since the intersection number $(2\overline{H}, 2\overline{H}, 2\overline{H})$ is equal to eight for a hyperplane $\overline{H}\subset\bP^3$, 
there exist no members of $\bar\fd$ which are singular at $\overline{P}_0$. 
Hence the condition (\ref{thm. conical 3}-a) holds. 
If three points $P_1,P_2,P_3$ are on a line $L\subset\bP^2$, and there is a cubic curve $C_3\subset\bP^2$ through the nine points $P_1,P_2,P_3,T_1,\dots,T_6$ which does not contain $\Delta_X$, 
then $\overline{P}_1,\overline{P}_2,\overline{P}_3$ are on two hyperplanes $\overline{H}_1, \overline{H}_2$ such that $p(\overline{H}_1\setminus\{\overline{P}_0\})=L$ and $\alpha(\overline{H}_2)=C_3$ by Theorem~\ref{thm. corr} (\ref{thm. corr 3}), 
which is a contradiction to the assumption that $\overline{P}_0,\dots,\overline{P}_7$ are in general position. 
Hence condition (\ref{thm. conical 3}-c) holds. 
If there exists a cubic curve $C_3$ through eleven points $P_1,\dots,P_5, T_1,\dots,T_6$, 
then $C_3$ contains $\Delta_X$ by Theorem~\ref{thm. corr} (\ref{thm. corr 3}), which is contradiction to condition (\ref{thm. conical 3}-a). 
Hence (\ref{thm. conical 3}-d) holds. 

Conversely, we suppose that condition (\ref{thm. conical 3}) holds. 
The correspondences $\alpha_1$ and $\alpha_2$ in Theorem~\ref{thm. corr} (\ref{thm. corr 3}) implies 
$\dim\bar\fd\geq2$ and that the eight points $\overline{P}_0,\dots,\overline{P}_7$ are in general position. 
By Proposition~\ref{prop. conical}, condition (\ref{thm. conical 1}) holds. 

We next prove that conditions (\ref{thm. conical 1}) and (\ref{thm. conical 2}) are equivalent. 
Assume that condition (\ref{thm. conical 1}) holds. Furthermore we can assume  $(\star)$ in the proof of Proposition~\ref{prop. 3,3 split} and that  $X$ is defined by 
\[ f_3^2-4\, f_1\,f_2=0, \]
where $f_1,\, f_2,\, f_3\in\bC[x,y,z,w]$ are homogeneous polynomials of degree two 
such that $f_1=f_2=f_3=0$ at $\overline{P}_0$. 
Hence $f_1, f_2$  and $f_3$ are of the form 
\[ f_1=a_1\, w+a_2 , \hspace{2em} f_2=b_1\, w+b_2, \hspace{2em} f_3=c_1\, w+c_2, \]
where $a_i, b_i, c_i\in\bC[x,y,z]$ are homogeneous polynomials of degree $i$ for $i=1,2$. 
Since $X$ has a node at $\overline{P}_0$, the linear forms $a_1,b_1,c_1$ are linearly independent over $\bC$. 
Hence we may assume that $a_1=x$, $b_1=y$, $c_1=z$. 
Then the conic $\Delta_X$ is given by $z^2-4\,x\,y=0$ by Theorem~\ref{thm. corr} (\ref{thm. corr 2}). 
In this case, the double cover $\pi_{\Delta_X}:\bP^1\times\bP^1\to\bP^2$ is given by 
$([s_1:s_2],[t_1:t_2]) \mapsto [s_1\, t_1 :  s_2\, t_2  : s_1\, t_2+s_2\, t_1]$. 
By direct computation, we can see that $\pi_{\Delta_X}^{\ast}\Gamma_X$ is defined by 
\[ (s_1^2\, \pi_{\Delta_X}^{\ast}b_2-s_1\, s_2\,\pi_{\Delta_X}^{\ast}c_2+s_2^2\,\pi_{\Delta_X}^{\ast}a_2)(t_1^2\,\pi_{\Delta_X}^{\ast}b_2-t_1\, t_2\,\pi_{\Delta_X}^{\ast}c_2+t_2^2\,\pi_{\Delta_X}^{\ast}a_1)=0. \]
Thus $\Gamma_X$ is a splitting curve of type $(2,4)$ with respect to $\Delta_X$. 

Conversely, we assume that $\pi_{\Delta_X}^{\ast}\Gamma_X=D^++D^-$, where $D^+$ is a $(2,4)$-curve on $\bP^1\times\bP^1$, and $D^-=\iota_{\Delta_X}^{\ast}D^+$. 
We put $P_1,\dots,P_7$ as the nodes of $\Gamma_X$ such that the image of $D^+\cap D^-$ under $\pi_{\Delta_X}$ is $\{ P_1,\dots,P_7,T_1,\dots,T_6 \}$, 
and let $\overline{P}_1,\dots,\overline{P}_7\in X$ be nodes of $X$ with $p(\overline{P}_i)=P_i$ for $i=1,\dots,7$. 
By Proposition~\ref{prop. conical}, it is sufficient to prove that $\overline{P}_0,\dots,\overline{P}_7$ are in general position and $\dim\bar\fd\geq 2$. 
By Corollary~\ref{cor. dimension} and Theorem~\ref{thm. corr} (\ref{thm. corr 3}), 
we obtain $\dim\bar\fd\geq 2$. 

Suppose that five points of $\overline{P}_0,\dots,\overline{P}_8$ are on a hyperplane $\overline{H}_0\subset\bP^3$. 
By Lemma~\ref{lem. 3,3 and 2,4}, we may assume $\overline{P}_1,\dots,\overline{P}_5$ and $\overline{P}_0\not\in\overline{H}_0$. 
Hence we may assume that $\overline{H}_0$ is defined by $w=0$. 
If $\overline{P}_6\in\overline{H}_0$, then $\overline{P}_1,\dots,\overline{P}_6$ are on a conic $\overline{C}$ on $\overline{H}_0$. 
Indeed, if $\overline{P}_6\in\overline{H}_0$, then either $\overline{P}_1,\dots,\overline{P}_6$ are on a conic on $\overline{H}_0$, or $\overline{H}_0$ is a fixed component of $\bar{\fd}$; 
if $\overline{H}_0$ is a fixed component of $\bar{\fd}$, then $\overline{P}_7\in\overline{H}_0$ since $\dim\bar{\fd}\geq 2$, hence $X|_{\overline{H}_0}=2\overline{C}$ for some conic $\overline{C}$ on $\overline{H}_0$ since $X|_{\overline{H}_0}$ is a quartic divisor on $\overline{H}_0$ with seven singular points. 
Thus if $\overline{P}_6\in\overline{H}_0$, then, by Proposition~\ref{prop. 3,3 split} and Lemma~\ref{lem. 3,3 and 2,4} (\ref{lem. 3,3 and 2,4 1}), $\Gamma_X$ is a splitting curve of type $(3,3)$, which is contradiction to Lemma~\ref{lem. 3,3 and 2,4} (\ref{lem. 3,3 and 2,4 2}). 
Therefore, $\overline{P}_6$ and $\overline{P}_7$ are not on $\overline{H}_0$. 

Let $\overline{C}\subset\overline{H}_0$ be the unique conic through the five points $\overline{P}_1,\dots,\overline{P}_5$. 
If all members of $\bar\fd$ are singular, then $\overline{C}$ is reducible and the singular point $\overline{P}$ of $\overline{C}$ is a singular point of all member of $\bar\fd$. 
In this case, the three points $\overline{P}_0$, $\overline{P}_6$ and $\overline{P}_7$ are collinear since $\dim\bar\fd\geq2$ and no four points of $P_1,\dots,P_7$ are collinear, 
which is a contradiction to that $X$ contains no lines through $\overline{P}_0$. 
Hence there is a smooth member $\overline{D}_0\in\bar\fd$. 
We put 
\[ \tilde{\fd}_1:=\{ \overline{D}|_{\overline{D}_0}-\overline{C} \mid \overline{D}\in\bar\fd\setminus\{\overline{D}_0\} \}. \]
Note that $\tilde{\fd}_1$ is a linear system on $\overline{D}_0\cong\bP^1\times\bP^1$ consisting of $(1,1)$-curves through $\overline{P}_0, \overline{P}_6, \overline{P}_7$, and $\dim\tilde{\fd}_1\geq 1$. 
Since the self intersection number of a $(1,1)$-curve is equal to two, 
$\tilde{\fd}_1$ has a fixed component, say $\overline{C}_1$. 
Since $\dim\tilde{\fd}_1\geq1$, $\overline{P}_0, \overline{P}_6,\overline{P}_7$ are on $\overline{C}_1$. 
This implies that the three points $\overline{P}_0, \overline{P}_6,\overline{P}_7$ are collinear in $\bP^3$, which is a contradiction. 
Therefore, no five points of $\overline{P}_0,\dots,\overline{P}_7$ are on a hyperplane of $\bP^3$. 

Suppose that three points of $\overline{P}_0,\dots,\overline{P}_7$ are collinear. 
We may assume that $\overline{P}_5,\overline{P}_6, \overline{P}_7$ are on a line $\overline{L}\subset\bP^3$. 
By the assumption and Lemma~\ref{lem. 3,3 and 2,4}, we have $\overline{P}_0,\dots,\overline{P}_4\not\in\overline{L}$. 
Moreover, by the above argument, we may assume that there are no hyperplanes in $\bP^3$ through $\overline{P}_0,\dots,\overline{P}_4$.
Note that all member of $\bar\fd$ contain $\overline{L}$, and there is an irreducible member in $\bar\fd$ since $\dim\bar\fd\geq 2$. 

\textit{Case 1.} \ 
Suppose that there is a smooth member $\overline{D}_0\in\bar\fd$. 
We put 
\[ \tilde{\fd}_2:=\{\overline{D}|_{\overline{D}_0}-\overline{L}\mid \overline{D}\in\bar\fd\setminus\{\overline{D}_0\}\}. \]
We may assume that $\overline{L}$ is a $(0,1)$-curve on $\overline{D}_0$, and 
 that $\tilde{\fd}_2$ consists of $(2,1)$-curves through $\overline{P}_0,\dots,\overline{P}_4$. 
The linear system $\bar{\fd}_2$ has a fixed component $\overline{C}_2$ since $\overline{C}^2=4<5$ for $\overline{C}\in\tilde{\fd}$. 
Since no hyperplanes in $\bP^3$ pass through $\overline{P}_0,\dots,\overline{P}_4$, 
$\overline{P}_0,\dots,\overline{P}_4$ are on a $(2,0)$-curve on $\overline{D}_0$. 
This implies that five points of $\overline{P}_0,\dots,\overline{P}_7$ are on a hyperplane, 
which is a contradiction. 

\textit{Case 2.} \ 
Suppose that all member of $\bar\fd$ are singular. 
Let $\overline{D}_1$ be an irreducible member of $\bar\fd$. 
Let $\Sigma_2\to\overline{D}_1$ be the blowing-up at the singular point of $\overline{D}_1$, 
and $\sigma:\Sigma_2\to\bP^3$ be the composition of the blowing-up $\Sigma_2\to\overline{D}_1$ and the inclusion $\overline{D}_1\subset\bP^3$. 
We put $F\subset\Sigma_2$ as the strict transformation of $\overline{L}$, and 
\[ \tilde{\fd}_3:=\{\sigma^{\ast}\overline{D}-F\mid \overline{D}\in\bar\fd\setminus\{\overline{D}_1\}\}. \]
Note that $\tilde{\fd}_3$ consists of curves linearly equivalent to $3F+2\, S_0$ and passes through $\sigma^{-1}(\{\overline{P}_0,\dots,\overline{P}_4\})$, 
where $S_0$ is the section of $\Sigma_2$ with $S_0^2=-2$. 
Since $(3\,F+2\,S_0)^2=4$, $\tilde{\fd}_3$ has a fixed component. 
Let $C$ be an irreducible fixed component of $\tilde{\fd}_3$. 
By $\dim\tilde{\fd}_3\geq 1$ and \cite[V Corollary~2.18]{hartshorne}, $C$ is linearly equivalent to $F$ or $2\,F+S_0$. 
If $C$ linearly equivalent to $2\,F+S_0$, then $\sigma^{-1}(\{\overline{P}_0,\dots,\overline{P}_4\})\subset C$ since $\dim\tilde{\fd}\geq 1$, 
hence $\overline{P}_0,\dots,\overline{P}_4$ are on a hyperplane by Lemma~\ref{lem. hirzebruch}, which is a contradiction. 
If $C$ is linearly equivalent to $F$, then $C$ through at least two points of $\sigma^{-1}(\{\overline{P}_0,\dots,\overline{P}_4\})$, say $\sigma^{-1}(\overline{P}_3)$ and $\sigma^{-1}(\overline{P}_4)$. 
This implies that $\overline{P}_3,\dots,\overline{P}_7$ are on a hyperplane, which is a contradiction. 
Hence, no three of $\overline{P}_0,\dots,\overline{P}_7$ are collinear. 
Thus condition (\ref{thm. conical 1}) holds. 
\end{proof}

\section{Examples}\label{sec. examples}

In this section, we give examples of irreducible $6$ and $7$-nodal sextic curves. 
Consequently, we have examples of Zariski pairs. 
We first explain a method of construction of a non-splitting sextic curve with respect to a contact conic. 
\medskip

\textit{A method of construction of a non-splitting sextic.} \ 
Let $\overline{P}_0,\dots,\overline{P}_5$ be six points of $\bP^3$ in general position, 
and let $\bar{\fd}$ be the linear system $|2\overline{H}-\overline{P}_0-\dots-\overline{P}_5|$ on $\bP^3$, where $\overline{H}$ is a hyperplane. 
Note that $\dim\bar{\fd}=3$. 
We consider the rational map $\Phi_{\bar{\fd}}:\bP^3\dashrightarrow\bP^3$ given by $\bar{\fd}$. 
Let $\sigma:\widetilde{\bP}^3\to\bP^3$ be the blowing-up at the six points $\overline{P}_0,\dots,\overline{P}_5$. 
Then $\Phi_{\bar{\fd}}$ induces a generically finite morphism of degree two $\widetilde{\Phi}_{\bar{\fd}}:\widetilde{\bP}^3\to\bP^3$.  
The pull-back of a general quadric surface in $\bP^3$ by $\widetilde{\Phi}_{\bar{\fd}}$ gives an irreducible quartic surface in $\bP^3$ with six nodes at $\overline{P}_0,\dots,\overline{P}_5$. 
If $\overline{D}_0\subset\bP^3$ is a certain quadric cone, 
then the image of $\widetilde{\Phi}_{\bar{\fd}}^{\ast}\overline{D}_0$ by $\sigma$ is a syzygetic quartic surface. 
If we chose certain smooth conics $\overline{D}_1, \overline{D}_2\subset\bP^3$ tangent to the branch locus of $\widetilde{\Phi}_{\bar{\fd}}$ at one and two points, respectively, 
then $X_1:=\sigma(\widetilde{\Phi}_{\bar{\fd}}^{\ast}\overline{D}_1)$ and $X_2:=\sigma(\widetilde{\Phi}_{\bar{\fd}}^{\ast}\overline{D}_2)$ are asyzygetic quartic surfaces with seven and eight nodes, respectively. 
In general, $X_1$ and $X_2$ give non-splitting sextic curves $\Gamma_{X_1}$ and $\Gamma_{X_2}$ with respect to $\Delta_{X_1}$ and $\Delta_{X_2}$, respectively. 
\medskip

We construct the non-splitting sextic curves in Example~\ref{ex. 6-nodal non-split 2} and \ref{ex. 7-nodal non-split} by the above method. 
However, we omit to describe the morphism $\Phi_{\bar{\fd}}$ and smooth quadric surfaces $\overline{D}_1$, $\overline{D}_2$ explicitly. 

Let $\Delta\subset\bP^2$ be a smooth conic. 
Let $(s:t)$ and $(u:v)$ be systems of homogeneous coordinates of $\bP^1$, 
and let $(x:y:z)$ be one of $\bP^2$. 
After certain projective transformation, we may assume that $\Delta$ is given by $\delta_2=0$, where $\delta_2=z^2-4\,xy$, and 
the double cover $\pi_{\Delta}:\bP^1\times\bP^1\to\bP^2$ is described as $\pi_{\Delta}(s:t, u:v)=(su:tv:sv+tu)$. 

\begin{Rem}\label{rem. irred}\rm
Let $\Gamma$ be a $r$-nodal sextic curve for a positive integer $r\leq 7$. 
If $\Gamma$ is reducible, then $\Gamma$ consists of a line and an irreducible curve of degree $5$ by B\'ezout's theorem. 
Conversely, if $\Gamma$ is irreducible, then no $5$ nodes of $\Gamma$ are collinear by B\'ezout's theorem.   
Hence $\Gamma$ is irreducible if and only if no $5$ nodes of $\Gamma$ are collinear. 
\end{Rem}

\subsection{$6$-nodal sextics}

Let $\Gamma_6$ be a $6$-nodal sextic curve on $\bP^2$ such that $\Delta$ is a contact-conic of $\Gamma_6$. 
By Corollary~\ref{lem. number of nodes}, 
$\Gamma_6$ is either a splitting curve of type $(3,3)$ or 
non-splitting curve with respect to $\Delta$. 

\begin{Ex}\label{ex. 6-nodal split}{\rm
Let $\Gamma_6\subset\bP^2$ be a splitting $6$-nodal sextic curve with respect to $\Delta$. 
From the proof of Theorem~\ref{thm. split}, $\Gamma_6$ is given by 
$c_3^2-\delta_2\,c_2^2=0$ for some $c_i\in H^0(\bP^2,\calO_{\bP^2}(i))$. 

For example, if $c_2=x y+y z+z x$ and $c_3=x^3+y^3+z^3$, 
then $\Gamma_6$ is given by 
\[ (x^3+y^3+z^3)^2-(z^2-4\,x y)(x y+y z+z x)^2=0, \]
and $\Gamma_6$ is a $6$-nodal sextic curve. 
The nodes of $\Gamma_6$ are $(\alpha_i : -\alpha_i^5-2\,\alpha_i^4-\alpha_i^3-3\,\alpha_i-1 : 1)$ ($i=1,\dots,6$), 
where $\alpha_i$ are the roots of $\alpha^6+3\,\alpha^5+3\,\alpha^4+\alpha^3+3\,\alpha^2+3\,\alpha+1=0$. 
Note that the $6$ nodes of $\Gamma_6$ are intersection of smooth conic and cubic given by 
\[ x^3+y^3+z^3=xy+yz+zx=0. \]
By Remark~\ref{rem. irred}, $\Gamma_6$ is irreducible. 
}\end{Ex}

\begin{Ex}\label{ex. 6-nodal non-split 1}{\rm
Let $L_1,\dots,L_4\subset\bP^2$ be general lines given by equations $l_i=0$, and 
let $C_3$ be a cubic curve given by $c_3=0$. 
If the $6$-nodes $P_1,\dots,P_6$ of $L_1+\dots+L_4$ are not on $\Delta$ and $C_3$ passes through $P_1,\dots,P_6$, 
then the sextic curve $\Gamma'_6$ given by $c_3^2-\delta_2\, l_1\, l_2\, l_3\, l_4=0$ has $6$ nodes at $P_1,\dots,P_6$. 

For example, if $l_1=z,\, l_2=x-y,\, l_3=2\,x-y,\, l_4=x+y-2\,z$ and 
$c_3=2\,x^3-x^2\,y+3\,x^2z-2\,x y^2-4\,x z^2+y^3+y z^2$, 
then $\Gamma'_6$ is given by 
\[ (2\,x^3-x^2y+3\,x^2z-2\,x y^2-4\,x z^2+y^3+y z^2)^2-z(x-y)(2\,x-y)(x+y-2\,z)(z^2-4\,x y)=0, \]
and $\Gamma'_6$ is a $6$-nodal sextic curve. 
The nodes of $\Gamma'_6$ are $(1:1:0), (1:2:0), (1:-1:0), (0:0:1), (1:1:1), (2:4:3)$. 
By Theorem~\ref{thm. split}, $\Gamma'_6$ is a non-splitting curve with respect to $\Delta$ since the $6$ nodes of $\Gamma'_6$ are not on a conic. 
Moreover, $\Gamma'_6$ is irreducible since no $5$ nodes of $\Gamma'_6$ are collinear. 
}\end{Ex}

The following example is a non-splitting $6$-nodal sextic curve with respect to $\Delta$, 
whose nodes are in general position. 

\begin{Ex}\label{ex. 6-nodal non-split 2}{\rm
Put $c''_3$ and $c_4''$ as follows; 
\begin{eqnarray*}
c''_3&=&
3328\,{x}^{3}+1392\,{x}^{2}y-672\,{x}^{2}z+180\,x{y}^{2} \\ 
&&-516\,xyz-180
\,x{z}^{2}+10\,{y}^{3}-33\,{y}^{2}z-45\,y{z}^{2} \\
c''_4&=&
10496\,{x}^{4}+6272\,{x}^{3}y-2528\,{x}^{3}z+1200\,{x}^{2}{y}^{2}-912
\,{x}^{2}yz+1176\,{x}^{2}{z}^{2} \\
&&+80\,x{y}^{3}-90\,x{y}^{2}z+246\,xy{z}
^{2}+2\,{y}^{4}-5\,{y}^{3}z+15\,{y}^{2}{z}^{2}
\end{eqnarray*} 
Let $\Gamma''_6$ be the sextic curve given by the following equation; 
\[ (c_3'')^2-432\,\delta_2 c_4''=0 \]
Then $\Gamma''_6$ is a $6$-nodal curve with a simple contact conic $\Delta$. 
The nodes of $\Gamma''_6$ are $(0:0:1), (0:3:2), (-1:4:0), (-1:10:6), (-1:16:8), (-3:36:38)$. 
Hence $\Gamma''_6$ is irreducible. 
It easy to see that there are no conic passing through the $6$ nodes of $\Gamma''_6$. 
Hence $\Gamma''_6$ is non-splitting curve with respect to $\Delta$ by Theorem~\ref{thm. split}. 
}\end{Ex}

We have an example of Zariski pair. 

\begin{Th}
Let $\Gamma_6$, $\Gamma_6'$ and $\Gamma_6''$ be the $6$-nodal sextics in Example~\ref{ex. 6-nodal split}, \ref{ex. 6-nodal non-split 1} and \ref{ex. 6-nodal non-split 2}, respectively. 
Then the pairs $(\Gamma_6+\Delta,\,\Gamma_6'+\Delta)$ and $(\Gamma_6+\Delta,\,\Gamma_6''+\Delta)$ are Zariski pairs. 
\end{Th}

\begin{Rem}{\rm
It is not known whether $(\Gamma_6'+\Delta,\,\Gamma_6''+\Delta)$ is a Zariski pair, or not. 
}\end{Rem}

\subsection{7-nodal sextics}

The genus of a $7$-nodal sextic curve is equal to $3$. 
Hence all $7$-nodal sextic curve are not splitting curves of type $(1,5)$ with respect to any contact conic. 

\begin{Ex}\label{ex. 7-nodal type (3,3)}{\rm
We put 
\begin{eqnarray*}
c_2 &=& {z}^{2}-4\,xy,  \\[0.5em]
c_3 &=& {y}^{2}z-3\,xyz+{z}^{3}-{x}^{2}z,  \\[0.5em]
c_4 &=& \left( {z}^{2}-xy-{y}^{2}+{x}^{2} \right) ^{2}.
\end{eqnarray*}
Let $\Gamma_7$ and $\Delta$ be the sextic curve and the conic given by $c_3^2-c_2c_4=0$ and $c_2=0$, respectively. 
Then $\Gamma_7$ is a $7$-nodal sextic with  a simple contact conic $\Delta$, and 
the $7$ nodes of $\Gamma_7$ are $(0:0:1)$, $(\alpha_i:2\alpha_i^3+\alpha_i:1)$ and  $(\beta_j:1:0)$ for $i=1,\dots, 4$ and $j=1,2$, 
where $\alpha_i$ and $\beta_j$ are roots of the equations $4\alpha^4+2\alpha^2-1=0$ and $\beta^2-\beta-1=0$, respectively. 
The $6$ nodes $(\alpha_i:2\alpha_i^3+\alpha_i:1)$ and $(\beta_j:1:0)$ are on a smooth conic $z^2-xy-y^2+x^2=0$, hence $\Gamma_7$ is irreducible. 
By Theorem~\ref{thm. split}, $\Gamma_7$ is a splitting curve of type $(3,3)$ with respect to $\Delta$. 
}\end{Ex}

\begin{Ex}\label{ex. 7-nodal type (2,4)}{\rm
We put 
\begin{eqnarray*}
f_1 &=& xw-y^2+z^2, \\[0.5em]
f_2 &=& yw-x^2+z^2, \\[0.5em]
f_3 &=& zw-x^2+y^2. 
\end{eqnarray*}
Then $f_3^2-4\,f_1f_2=0$ defines a syzygetic quartic surface $X$ in $\bP^3$ with $8$ nodes at $(0:0:0:1)$, $(0:1:1:-1)$, $(-1:0:1:1)$, $(1:1:0:1)$, $(1:1:1:0)$, $(-1:1:1:0)$, $(1:-1:1:0)$, $(1:1:-1:0)$. 
We have the $7$-nodal sextic curve $\Gamma_X$ with the simple contact conic $\Delta_X$ from the projection $X\dashrightarrow\bP^2$ with center $(0:0:0:1)$. 
Since the $8$ nodes of $X$ are in general position, $\Gamma_X$ is irreducible. 
By Theorem~\ref{thm. conical}, $\Gamma_X$ is a splitting curve of type $(2,4)$ with respect to $\Delta_X$. 
}\end{Ex}

\begin{Ex}\label{ex. 7-nodal non-split}\rm
We put 
\begin{eqnarray*}
c'_2 &=& -61\,{x}^{2}+20\,xy+4\,xz+4\,{y}^{2}-4\,yz+{z}^{2}, \\
c'_3 &=& -13\,{x}^{2}y+168\,{x}^{2}z-74\,xyz-8\,x{z}^{2}-8\,{y}^{2}z+7\,y{z}^{2}, \\
c'_4 &=& {x}^{2}{y}^{2}+16\,{x}^{2}yz-112\,{x}^{2}{z}^{2}-4\,x{y}^{2}z+64\,xy{z}^{2}-{y}^{2}{z}^{2}. 
\end{eqnarray*}
Let $\Gamma'_7$ and $\Delta'$ be the sextic curve and conic defined by ${c'_3}^2-4c'_2c'_4=0$ and $c'_2=0$, respectively. 
The sextic curve $\Gamma'_7$ is a $7$-nodal sextic with the simple contact conic $\Delta'$. 
The $7$ nodes of $\Gamma'_7$ are $(0:0:1)$, $(0:1:0)$, $(1:0:0)$, $(1:1:1)$, $(1:-2:1)$, $(-1:6:3)$, $(1:2:-3)$. 
Hence $\Gamma'_7$ is irreducible. 
It is easy to see that there are no conics passing through $6$ points of the $7$ nodes of $\Gamma'_7$. 
Hence $\Gamma'_7$ is not a splitting curve of type $(3,3)$ with respect to $\Delta'$  by Theorem~\ref{thm. split}. 
Moreover, we can check by direct computation that the dimension of the linear system consisting of quartic curves passing through the $7$ nodes and $6$ tangent points of $\Gamma'_7$ and $\Delta'$ is equal to $1$. 
Thus $\Gamma'_7$ is not a splitting curve of type $(2,4)$ with respect to $\Delta'$. 
Therefore, $\Gamma'_7$ is a non-splitting curve with respect to $\Delta'$.  
\end{Ex}

We have an example of Zariski $3$-plet. 

\begin{Th}
Let $\Gamma_7+\Delta$, $\Gamma_X+\Delta_X$ and $\Gamma'_7+\Delta'$ be the curves of Example~\ref{ex. 7-nodal type (3,3)}, \ref{ex. 7-nodal type (2,4)} and \ref{ex. 7-nodal non-split}, respectively. 
Then the $3$-plet $(\Gamma_7+\Delta, \Gamma_X, \Delta_X, \Gamma'_7+\Delta')$ is a Zariski $3$-plet. 
\end{Th}

\end{document}